\documentclass[reqno,11pt]{amsart}

\usepackage[foot]{amsaddr}
\usepackage{amsmath, amsthm, amsfonts, amssymb}
\usepackage{graphicx}
\usepackage{subfig}
\usepackage{hyperref}
\usepackage[capitalise,nameinlink]{cleveref}
\numberwithin{equation}{section}
\usepackage[T1]{fontenc}

\usepackage{xcolor}
\hypersetup{
    colorlinks,
    linkcolor={blue!50!blue},
    citecolor={red!20!blue},
    urlcolor={red!20!blue}
}

\usepackage[margin=1.5in]{geometry}
\usepackage{enumitem}

\newcommand{\xspace}{{\mathcal{X}_{space}}}
\newcommand{\xobst}{{\mathcal{X}_{obst}}}
\newcommand{\xfree}{{\mathcal{X}_{free}}}
\newcommand{\xobddry}{{\partial\xobst}}

\newcommand{\uspace}{\mathcal{U}_{space}}

\newcommand{\F}[1]{\mathcal{F}_{#1}}
\newcommand{\R}[1]{\mathcal{R}_{#1}}
\newcommand{\B}[1]{\mathcal{B}_{#1}}

\newcommand{\clr}{\textsc{clr}}
\newcommand{\wit}{\textsc{wit}}

\newtheorem{theorem}{Theorem}[section]
\newtheorem{proposition}[theorem]{Proposition}
\newtheorem{lemma}[theorem]{Lemma}
\newtheorem{corollary}[theorem]{Corollary}
\newtheorem{ex}[theorem]{Example}

\newtheorem{remark}[theorem]{Remark}
\newtheorem{ass}[theorem]{Assumption}

\newtheorem{definition}[theorem]{Definition}

\crefname{equation}{}{}

\title[Discontinuities in Clearance]{Kinodynamic Control Systems and
Discontinuities in Clearance}

\author[N. Armstrong]{Niles Armstrong}
\email[A1]{armstrongn@msoe.edu}

\author[J. Denny]{Jory Denny}
\email[A2]{joryldenny@gmail.com}

\author[J. LeCrone]{Jeremy LeCrone}
\email[A3]{jlecrone@richmond.edu}

\keywords{Control theory, Minimal time function, Discontinuous value function}

\subjclass[2010]{
93C10, 
49N60, 
93B03  
}

\begin{document}
\maketitle

\begin{abstract}
We investigate the structure of discontinuities in clearance (or minimum time) functions 
for nonlinear control systems with general, closed obstacles (or targets). 
We establish general results regarding interactions between admissible trajectories 
and clearance discontinuities:
e.g. instantaneous increases in clearance when passing through a discontinuity, 
and propagation of discontinuity along optimal trajectories. 
Then, investigating sufficient conditions for discontinuities, 
we explore a common directionality condition for velocities 
at a point, characterized by strict positivity of the minimal Hamiltonian.
Elementary consequences of this common directionality assumption are explored
before demonstrating how, in concert with corresponding obstacle configurations,
it gives rise to clearance discontinuities both on the surface of the obstacle and
propagating out into free space.
Minimal assumptions are made on the topological structure of obstacle sets.
\end{abstract}

\section{Introduction}

This paper studies nonlinear optimal control problems in constrained
environments and fine properties of associated clearance functions
(or minimal time functions in time--optimal settings).
We focus on settings where discontinuities in clearance arise
and investigate the fine analytic structure of these sets of discontinuities
when systems admit a strict directionality of admissible velocities.

As motivation for this investigation, we note that optimal control is relevant 
to questions in robotics, mechanical engineering, and aerospace engineering. 
Specifically, algorithms that plan the motion of dynamical systems from a 
start state to a goal region (motion planning algorithms) may exploit 
many of the properties of discontinuities we explore herein.
As an example of one such class of algorithms, Rapidly--exploring Random 
Trees (RRTs) build approximate representations of a robot's state space that 
encode feasible pathways for a robot to move 
through regions of a constrained state space \cite{l-pa-06}.
These methods seek not only to find the existence of feasible paths, but paths 
with cost optimality~\cite{kf-sbaomp-11} (e.g., minimal time or shortest distance) 
and safety guarantees~\cite{dgta-marrt-14} to ensure reliable clearance from obstacles. 

Although clearance is traditionally interpreted as a robot's geometric 
distance from obstacles (with respect to canonical metrics in state space) 
we embrace a system--dependent perspective on clearance.
As observed in applications (c.f., \cite{sjp-osbmpudc-15}), the dynamic 
limitations of a robot's motion should also receive proper consideration 
in questions of optimality and safety of admissible trajectories.
Such considerations lead one naturally to a formulation of clearance that
ignores obstacles outside the accessible region of a robot 
(c.f., \Cref{Def:CLRetc}), which corresponds with 
the minimal time function (in time--optimal settings) studied extensively 
in the mathematical control theory community \cite{AF20,CMN15,CMW12,CS15,CN10,FN15,N16,NS07,V97,WZ98}.

The regularity of clearance / minimal--time functions is an active 
area of research, with a large selection of literature dedicated to 
sufficient conditions ensuring regularity conditions near obstacles.
As a small sampling of this literature, we point the interested reader to
a selection of references discussing differentiability \cite{CS15},
continuity and semicontinuity \cite{WZ98}, and semiconcavity \cite{CMW12}
of minimal--time functions.
A ubiquitous assumption in these studies is the controllability of 
the system in a neighborhood of the obstacle set, classically exemplified by 
the Petrov condition (c.f., \cite{CLSW95}). 
Essentially, this condition enforces the existence of admissible 
velocities at all obstacle boundary points so that trajectories 
reasonably penetrate the obstacle.
Our investigation diverges from those efforts, which we mention 
primarily to highlight the fact that we will be working in settings
devoid of such Petrov--type controllability conditions.

It is well--known that discontinuities arise in optimal control settings 
(c.f., \cite{BP87,BP07,CQS97,FPT11}), 
though there has been little attention paid to the precise 
analytic structure of these discontinuities in clearance.
This paper presents a framework for this analysis along with a number of 
important initial results to be built upon in future investigations.
Further, as motivated above, we note that this work can potentially 
inform future developments in robot motion planning.

We proceed with a brief outline of the paper.
In \Cref{sec:DefAss}, we introduce the setting, assumptions, 
and definitions for the investigation.
We observe that the optimal cost--distance produces a quasi-metric on state space
(a form of asymmetric distance between states) associated to which we
identify quasi-metric cost--balls centered at a given state.
The geometric properties of these evolving cost--balls plays an important 
role in the analysis of discontinuities in clearance.

In \Cref{sec:PosHam}, we isolate an assumption of directionality
among admissible velocities -- quantified by strict positivity of the 
minimal Hamiltonian \Cref{Eqn:MinimalHamiltonian} -- 
in the direction of some vector $\xi.$
We state and prove a number of elementary consequences of this
strict directionality, including a type of 
small--time--local--non--returnability (\Cref{lem:STLNR})
and existence of persistent boundary points 
(\Cref{thm:PersistentBddry}) 
on corresponding reachable sets.
Embodied in this persistent boundaries result is an investigation
of a local envelope of propagating reachable sets, a concept that is
central to the analysis of discontinuities of clearance in free space.

In \Cref{sec:CLRProps}, we establish intrinsic 
properties of clearance; i.e. properties of clearance 
as one traverses admissible trajectories. 
Notably, in \Cref{thm:ClrIncreaseAlongTrajectories}, we
confirm that clearance cannot discontinuously decrease along 
admissible trajectories.

Finally, in \Cref{sec:CLRDisc} we study properties of 
discontinuities in clearance.
First, we characterize all discontinuities in
free space (\Cref{thm:DiscontinuousOnE}) as envelope 
points on the boundaries of multiple members of a family of  
propagating reachable waves (\Cref{Def:WaveEnvelope}).
These envelope points are further shown 
to form a continuous structure in space (\Cref{thm:EAlongTrajectories}),
propagated along optimal trajectories back to the obstacle set.
Next, we turn our attention to clearance discontinuities  
present on the boundary of the obstacle set itself. 
After a brief exploration of general properties, we focus
specifically on a type of envelope generator discontinuity (\Cref{def:EnvGen}); 
from which free space discontinuities propagate with arbitrarily 
small clearance.

Our main result is \Cref{thm:EnvGenThm},
providing a set of sufficient conditions ensuring
an obstacle boundary point is in fact an envelope generator.
Thus ensuring the existence of free space discontinuities nearby.
This result is preceded by motivating examples and 
followed by applications of the result.

\section{Definitions and Assumptions}\label{sec:DefAss}

\subsection{General Setting}
Let $\xspace \subseteq \mathbb{R}^n$ be state space, 
with inner product denoted $\langle \cdot, \cdot \rangle,$ 
and let $\uspace \subseteq \mathbb{R}^m$ be control space.
We fix a state function $f:\xspace \times \uspace \to \xspace,$ 
with regularity conditions to be addressed below.
Given two states $x,y \in \xspace,$ we denote by $\Pi_{u}(x,y)$
the collection of absolutely continuous trajectories 
$\pi: [0,T] \to \xspace$ solving the {\bf parameterized} control problem
\begin{equation}\label{Eqn:MoveXtoY}
\begin{cases}
    \dot{\pi}(t) = f(x(t),u(t)) & \text{a.e. $t \in [0,T]$}\\
    u(t) \in \uspace & \text{a.e. $t \in [0,T]$}\\
    \pi(0) = x, \\ \pi(T) = y,
\end{cases}
\end{equation}
for some $T = T_\pi \ge 0.$
We say that $\pi \in \Pi_{u}(x,y)$ is an {\bf unconstrained} trajectory
moving $x$ to $y$ in $T_\pi$ units of time. 
Further, we fix a continuous running cost function $\psi: \xspace^2 \to (0,\infty),$
with which one computes the cost to traverse $\pi \in \Pi_u(\cdot,\cdot)$ as
\begin{equation}\label{Eqn:CostIntegral}
    c_\pi := \int_0^{T_\pi} \psi(\pi(t), \dot{\pi}(t)) dt.
\end{equation}

To address the regularity conditions for the state function, we shift our
discussion to the equivalent {\bf nonparameterized} control system.
Namely, we define the admissible velocity multifunction (i.e. set--valued function)
\begin{equation}\label{Eqn:VelFctn}
    F: \xspace \rightrightarrows \xspace \qquad \text{with} \qquad 
    F(x) := \{ f(x,u): u \in \uspace\}.
\end{equation} 
Under mild assumptions, it is well known that $\Pi_{u}(x,y)$ coincides with 
the absolutely continuous solutions to the differential inclusion
\begin{equation}\label{Eqn:MoveXtoYInclusion}
\begin{cases}
    \dot{\pi}(t) \in F(\pi(t)) &\text{for a.e. $t \in [0,T]$}\\
    \pi(0) = x, \\ \pi(T) = y.
\end{cases}
\end{equation}
We direct the interested reader to standard introductory texts in mathematical
control theory \cite{BP07, CLSW98} for further details on this and
other standard results.

Finally, we adopt the following standard assumptions on velocity sets,
which (albeit indirectly) address assumptions one can make for state functions $f.$
\begin{ass}[Standing Hypotheses for Velocity Sets]\label{ass:VelSetAssumptions} \mbox{ }
\begin{enumerate}
    \item[{\bf (SH1)}] {\it (Nonempty, closed, bounded, convex velocity sets)}
        For all $x \in \xspace,$ the set $F(x) \subset \xspace$ is nonempty and convex,
        with graph $gr F := \{(x,v): v \in F(x)\}$ closed in $\xspace^2$, and for each compact
        set $K \subset \xspace,$ there exists a constant $M > 0$ so that
        $\sup \{ \|v\|: x \in K, v \in F(x)\} \le M.$
    \item[{\bf (SH2)}] {\it (Linear growth condition)} There exist $\gamma, c > 0$ so that
    for all $x \in \xspace$ it holds that $\|v\| \le \gamma \|x\| + c$ whenever $v \in F(x).$
    \item[{\bf (SH3)}] {\it (Local Lipschitz regularity)} For each compact set $K \subset \xspace,$
        there exists $k > 0$ such that
        \[
            F(x) \subseteq F(y) + k\|x-y\|\overline{B_1(0)} \qquad \text{for all $x,y \in K,$}
        \]
        where $\overline{B_1(0)}$ denotes the closed unit ball centered at $0 \in \xspace.$
\end{enumerate}
\end{ass}

We also define the {\bf minimal (or lower) Hamiltonian} at a point $x \in \xspace,$ in
the direction of some vector $\xi \in \xspace,$ as
\begin{equation}\label{Eqn:MinimalHamiltonian}
    h_F(x,\xi) := \inf_{v \in F(x)} \langle v, \xi \rangle.
\end{equation}
Informally, we note that $h_F(x,\xi) > 0$ means all admissible velocities
at $x$ have a nontrivial positive component in the $\xi$ direction.

We introduce kinodynamic constraints on our control system by way of the 
following identification of obstacles in $\xspace.$ Let $\xobst \subset \xspace$ be any
closed {\bf obstacle} set and denote by $\xfree := \xspace \setminus \xobst$ the 
open {\bf free space} within whose closure trajectories are constrained.
In particular, with $\xobst$ and $\xfree$ fixed, we define the collection of
{\bf admissible trajectories} moving $x \in \xspace$ to $y \in \xspace$ as
\[
    \Pi(x,y) := \{ \pi \in \Pi_{u}(x,y): \pi(t) \in \overline{\xfree} \quad \text{for all $t \in [0,T_\pi]$} \}.
\]
Note that $\Pi(x,y) = \emptyset$ whenever $x$ or $y$ are in the interior of the obstacle
set, $\xobst^\circ.$

The following definitions lay the foundations for analysis in $\xspace,$ 
using the intrinsic system--dependent cost--distance between states.
We mention that our assumption of positive running cost (i.e. $\psi > 0$)
means that one could reformulate our problem to a simple minimal--time problem 
(i.e. with $\psi \equiv 1$ and $c_\pi \equiv T_\pi$) throughout (c.f. \cite[Remark 6.7]{BP07}).
We take advantage of this equivalence to support application of known results from the 
literature, but maintain a framework with general running--cost functions for 
accessibility of our results in applications.

\begin{definition}\label{Def:CLRetc} 
Given $x,y \in \xspace$ and $\rho > 0,$ we define the following:
\begin{enumerate}
    \item The control--system--dependent distance (or cost--distance) from $x$ to $y$
        \[
            d_c(x,y) := \begin{cases} 
            \inf\{c_\pi : \pi \in \Pi(x,y)\} & 
                \text{if $\Pi(x,y) \ne \emptyset$} \\
            \infty & \text{if $\Pi(x,y) = \emptyset.$}
            \end{cases}
        \]
    \item The forward reachable set of radius $\rho$
        \[
            \F{\rho}(x) := 
                \{ y \in \xspace: d_c(x,y) < \rho\}.
        \]
    \item The reverse reachable set of radius $\rho$
        \[
            \R{\rho}(x) := 
                \{ y \in \xspace: x \in \F{\rho}(y)\}.
        \]
    \item The clearance from $\xobst$
        \[
            \clr(x) := \inf_{y \in \xobst} d_c(x,y).
        \]
    \item The set of witness points
        \[
            \wit(x) := \{y \in \xobddry: d_c(x,y) = \clr(x)\}.
        \]
\end{enumerate}
\end{definition}

In general, the cost--distance function forms a quasi--metric on $\xspace.$
That is, $d_c(x,y) \ge 0$ with $d_c(x,y) = 0$ if and only if $x = y,$ and 
$d_c(x,y) \le d_c(x,z) + d_c(z,y)$ for all $x,y,z \in \xspace,$
but $d_c(x,y) \ne d_c(y,x)$ in general.
We introduce the quasi--metric {\bf cost balls,} 
\[
    \B{\rho}(x) := \F{\rho}(x) \cup \R{\rho}(x) \qquad 
    \text{for $x \in \xspace.$}
\]
We conclude this section by exploring the connections between the standard
metric induced by norm $\|\cdot\|_{\xspace}$ and the quasi--metric here introduced. 
To complement the quasi--metric ball around $x,$ we 
denote the standard metric ball of radius $r > 0$ centered at $x$ as
\[
    B_r(x) := \{ y \in \xspace: \|x-y\|_{\xspace} < r\}.
\]
(Throughout, we adopt the convention that roman letters $r, s, t$ denote
radii for metric balls, while Greek letters $\rho, \mu, \eta$ denote radii for cost balls.)


Now we list a number of elementary properties that are needed for future analysis.
The proofs for these facts follow from standard arguments, with 
references (or proof ideas) provided as appropriate.

\begin{proposition}[Properties of cost and clearance]\label{prop:CostProperties}
  Suppose $x, y \in \xspace$ 
  \begin{enumerate}
      \item \cite[Proposition 2.2(a)]{WZ98} The sets $\R{\rho}(x)$ and $\F{\rho}(x)$ 
      evolve continuously in $\rho,$ with respect to the Hausdorff distance.
  
      \item \cite[Proposition 2.4]{WZ98} If sequences $(x_n), (y_n) \subset \xspace$
      converge to $x$ and $y,$ respectively, and $\pi_n \in \Pi(x_n,y_n)$ exist with
      $c_{\pi_n} \to c,$ then there exists $\pi \in \Pi(x,y)$ with $c_{\pi} = c.$
      
      \item \cite[Proposition 2.6]{WZ98} If $\Pi(x,y) \ne \emptyset,$ then 
      there exists an optimal trajectory $\pi \in \Pi(x,y)$ with $c_\pi = d_c(x,y).$
      Moreover, if $\clr(x) < \infty,$ then $\wit(x) \ne \emptyset$ and, 
      for all $y \in \wit(x),$ there exists an optimal $\pi \in \Pi(x,y)$ 
      with $c_\pi = \clr(x).$
      
      \item \cite[Theorem 3.11]{CLSW95} The sets $\F{\rho}(x)$
      are locally Lipschitz continuous in $x,$ with respect to the Hausdorff distance.
      More precisely, we have that the forward attainable sets
      \[
        \mathcal{A}_{\rho}(x) := \{y \in \xspace: c_\pi = \rho
        \text{ for some } \pi \in \Pi(x,y)\}
      \]
      have local Lipschitz continuous dependence on $x.$
      
      \item Given $0<\rho<\mu,$ it holds that $\overline{\B{\rho}(x)} \subseteq \B{\mu}(x).$\\
      {\rm (This is a straightforward consequence of property (b) above)}
  \end{enumerate}
\end{proposition}


\section{Consequences of Positive Hamiltonian}\label{sec:PosHam}

As motivated in the introduction, a strict directionality of admissible velocities
is the primary driving force in our analysis of clearance discontinuities. 
We establish in this section some of the preliminary consequences of this assumption.

Throughout this section, we assume that $x, \xi \in \xspace$ are given with
$\xi \ne 0$ and satisfying the property
\begin{equation}\label{eqn:PosHam}
    h_F(x,\xi) := \inf_{v\in F(x)} \langle v, \xi \rangle > 0.
\end{equation}
Further, given any $r^\star > 0$ we define a point that is a geometric distance
of $r^\star$ away from $x$ in the direction of $\xi,$ namely
\[
    x^\star := x + \frac{r^\star}{\|\xi\|} \xi.
\]

\begin{proposition}\label{prop:HPosNhood}
Given $h_F(x,\xi) >0$ and $r^\star > 0,$ then there exists $R \in (0,r^\star)$ so that 
$h_F(y,x^\star - y) > \frac{1}{2}h_F(x,x^\star - x)$
for all $y \in \overline{B_R(x).}$
\end{proposition}

\begin{proof}
This is a straightforward consequence of standing hypotheses {\bf (SH1)} and {\bf (SH3).}
We include details of the proof for the reader's convenience.
    
By compactness of $\overline{B_{2r^\star}(x)},$ fix values $M > 0$ and $K > 0$ so that
\begin{equation}\label{eqn:VelBound}
    y \in \overline{B_{2r^\star}(x)} \quad \text{and}
        \quad v \in F(y) \quad \Longrightarrow \quad
        \| v \| \le M,
\end{equation}
and
\begin{equation}\label{eqn:LipzBound}
    y,z \in \overline{B_{2r^\star}(x)} \quad   
        \Longrightarrow \quad
        F(y) \subset F(z) + K \|y-z\|\overline{B_1(0)}.
\end{equation}    
Select any $y \in \overline{B_{2r^\star}(x)}$ and $w \in F(y).$ 
It follows from \Cref{eqn:LipzBound} that
\[
    w = v_w + K\|x - y\| \phi_w,
        \qquad \text{for some} \qquad v_w \in F(x), \ \phi_w
        \in \overline{B_1(0)}.
\]
Employing Cauchy--Schwarz inequality, we compute
\begin{align*}
    \langle w, x^\star - &y \rangle = 
        \langle v_w, x^\star - x\rangle + 
        \langle v_w, x - y\rangle + 
        K\|x-y\|\langle \phi_w,x^\star-y\rangle\\
    &\ge h_F(x,x^\star-x) -
        |\langle v_w,x-y\rangle| -
        K\|x-y\|
        |\langle \phi_w, x^\star-y\rangle|\\
    &\ge h_F(x,x^\star-x) - 
        \|v\|\|x-y\| - K\|x-y\|\|\phi_w\|\|x^\star-y\|\\
    &\ge h_F(x,x^\star-x) - 
        \|x-y\|\big(M + 3r^\star K\big).
\end{align*}
The claim thus follows with any selection of 
$0 < R < h_F(x,x^\star - x) / 2\big(M + 3r^\star K\big).$
To see that $R < r^\star,$ note $h_F(x,x^\star - x) \to 0$
as $x \to x^\star,$ and so $x^\star \notin \overline{B_R(x)}.$
\end{proof}

\begin{proposition}\label{prop:DirectPropagation}
Given $h_F(x,\xi) > 0$ and $r^\star > 0,$ there exists $t^\star > 0$
so that
\[
    \frac{d}{dt} \|x^\star - \pi(t)\| \le - \frac{h_F(x,x^\star - x)}{r^\star}
\]
for all maximally--defined trajectories\footnote{By maximally--defined in the unconstrained setting, 
we simply
consider trajectories that have been extended (as necessary) to a maximal
time interval; i.e. so that $T_\pi = \infty.$}
$\pi \in \Pi_u(x,\cdot)$ and a.e. $t \in [0,t^\star].$ 
\end{proposition}

\begin{proof}
Proceeding from the proof of \Cref{prop:HPosNhood},
we select any value 
\[
   0 < t^\star < \frac{R}{M}.
\]
Now, consider a maximally--defined trajectory 
$\pi \in \Pi_u(x,\cdot).$ Given $t \in [0,t^\star],$ first observe that
$\pi(t) \in B_R(x) \subset B_{2r^\star}(x)$ by \Cref{eqn:VelBound}, since
\[
    \|\pi(t)-x\|
        \le \int_0^{t} \|\dot{\pi}(s)\|ds
        \le M t < R < r^\star.
\]
Select an arbitrary vector $w \in F(\pi(t)).$
It follows from \Cref{eqn:LipzBound} that
\[
    w = v_w + K\|x - \pi(t)\| \phi_w,
        \qquad \text{for some} \qquad v_w \in F(x), \ \phi_w
        \in \overline{B_1(0)}.
\]
Thus, by \Cref{prop:HPosNhood}, we compute
\begin{equation}\label{eqn:PenetrationRate}
\begin{split}
    \frac{d}{dt}\|x^\star-\pi(t)\|^2 &= 
        \frac{d}{dt}\big\langle x^\star - \pi(t), x^\star - \pi(t) \big\rangle\\
    &\le -2h_F(\pi(t),x^\star-\pi(t)) \le -h_F(x,x^\star - x),
\end{split}
\end{equation}
noting the computation is valid for a.e. $t \in [0,t^\star],$ 
by absolute continuity of $\pi(\cdot).$
\end{proof}

The following two lemmas are consequences of the previous propositions.
The first result provides a quantitative statement for 
uniform directional propagation 
(\Cref{fig:RadiusOfPenetration} is a visualization of this phenomenon). 
The second lemma quantifies a type
of non--small--time--local--controllability
present whenever \Cref{eqn:PosHam} holds.

\begin{lemma}\label{lem:UnifPenetration}
Given $h_F(x, \xi) > 0$ and $r^\star > 0,$ set $t^\star > 0$ as in 
\Cref{prop:DirectPropagation}. Then, for every 
$t \in (0,t^\star],$ there exists $\eta^\star = \eta^\star(t) \in (0,1)$
so that
\[
    \pi(t) \in B_{\eta^\star r^\star}(x^\star)
\]
for all maximally--defined trajectories $\pi \in \Pi_u(x,\cdot).$
\end{lemma}

\begin{proof}
Proceeding from \Cref{prop:DirectPropagation}, we compute
\[
    \|x^\star - \pi(t)\|^2 = (r^\star)^2 + 
        \int_0^{t} \frac{d}{ds} \|x^\star-\pi(s)\|^2 ds
        \le (r^\star)^2 - t h_F(x,x^\star-x).
\]
It follows that $\|x^\star - \pi(t)\| < \eta^\star r^\star = \eta^\star \|x^\star - x\|$
for any choice of parameter
\[
    1 > \eta^\star > \left(1-\frac{t h_F(x,x^\star-x)}{(r^\star)^2}\right)^{1/2}.  \qedhere
\]
\end{proof}

\begin{lemma}\label{lem:STLNR}
Suppose $x \in \xfree$ and $\xi \in \xspace$ with $h_F(x,\xi) > 0.$ 
Then there exists $\rho^\star = \rho^\star(x) > 0$ so that for all 
$\rho \in (0, \rho^\star),$ there exists $r(\rho) = r(\rho,x) > 0$ with 
\[
    \B{\rho^\star}(x) \setminus \B{\rho}(x) \subset \big(B_{r(\rho)}(x)\big)^c.
\]
Equivalently, if $y \in \xspace$ with $\rho \le \min\{ d_c(x,y), d_c(y,x)\} < \rho^\star,$
then $\|x - y\| \ge r(\rho) = r(\rho,x).$
\end{lemma}

\begin{proof}
By $\xfree$ open, we select $r^\star > 0$ so that $B_{r^\star}(x) \subset \xfree.$ 
Now, we fix $R > 0$ as in \Cref{prop:HPosNhood},
$t^\star > 0$ as in \Cref{prop:DirectPropagation},
and select $\rho^\star = \rho^\star(x) > 0$ sufficiently small that 
$\B{\rho^\star}(x) \subset B_R(x).$
By compactness of $\overline{B_R(x)},$ standing hypothesis {\bf (SH1)}, 
and continuity of $\psi,$ we define
\begin{equation}\label{eqn:PsiStar}
    \psi^\star := \max\{ \psi(y,v): y \in \overline{B_R(x)}, v \in F(y)\}.
\end{equation}

Now, let $\rho \in (0,\rho^\star),$ apply \Cref{lem:UnifPenetration} to set 
\[
    \eta^\star = \eta^\star\big(\min\left\{t^\star, \rho / \psi^\star \right\}\big)
\]
and define $r(\rho) = r^\star - \eta^\star r^\star.$

Consider any trajectory $\pi \in \Pi_u(x,\cdot) \cup \Pi_u(\cdot,x)$ with 
$\rho \le c_\pi < \rho^\star.$
We note that $\B{\rho^\star}(x) \subset \xfree,$ so all such unconstrained
trajectories are likewise admissible (constrained) trajectories.
It follows from \Cref{Eqn:CostIntegral} that $T_\pi \ge \frac{\rho}{\psi^\star},$ and 
we conclude the proof with the observation that  
\[
    \pi(T_\pi) \in B_{\eta^\star r^\star}(x^\star) \subset \big(B_{r(\rho)}(x)\big)^c. \qedhere
\] 
\end{proof}

\begin{figure}[tbhp]
\centering
\subfloat[\Cref{lem:UnifPenetration}]{
\includegraphics[clip=true,trim=0.4in 1in 0.4in 2in,height=2.5in]{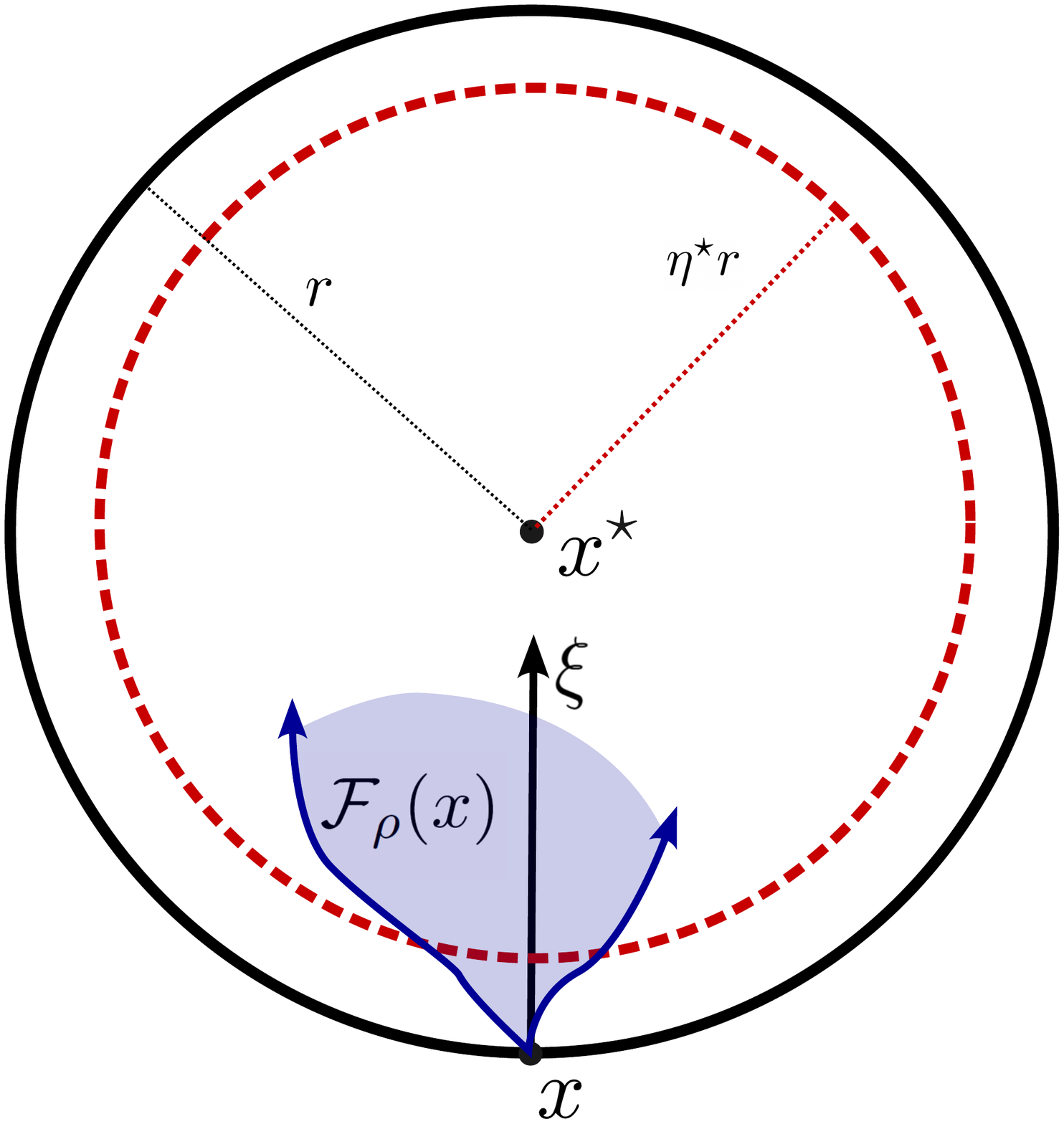}\label{fig:RadiusOfPenetration}}
\hspace{1cm}
\subfloat[\Cref{thm:PersistentBddry}]{
\includegraphics[clip=true,trim=.75in 1in 1in 2in,height=2.5in]{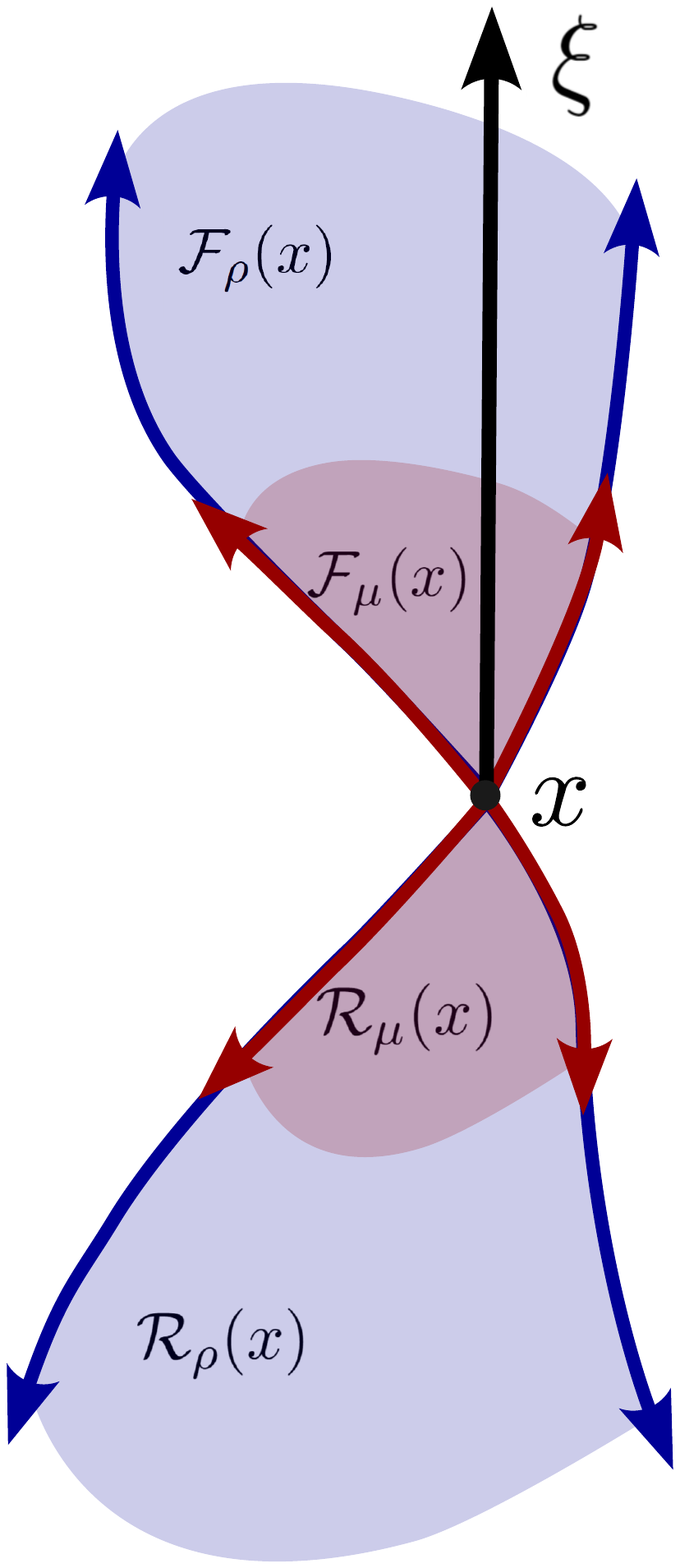}\label{fig:PersistentBoundary}}
\caption{(a) Visualizing \Cref{lem:UnifPenetration},
quantifying the uniform directional propagation of trajectories 
leaving a point with strictly positive minimal Hamiltonian.
(b) Visualizing \Cref{thm:PersistentBddry} (below). 
In the example displayed, {\bf all} boundary points of the smaller reachable
sets $\R{\mu}(x)$ and $\F{\mu}(x)$ persist as boundary points
of the larger sets $\R{\rho}(x)$ and $\F{\rho}(x).$}
\end{figure}

In the next result, we demonstrate how \Cref{eqn:PosHam}
gives rise to the existence of persistent boundary points on 
reachable sets (\Cref{fig:PersistentBoundary}).
The presence of such persistent boundary points, within 
families of propagating sets, plays a key role
in the analysis of clearance discontinuities in \Cref{sec:CLRDisc} below.

\begin{theorem}\label{thm:PersistentBddry}
Suppose $x \in \xfree$ and $\xi \in \xspace$ with $h_F(x,\xi) > 0.$
For all $r > 0$ there exists $\rho > 0$ so that
\[
    \big(\partial\R{\rho}(x) \cap \partial\R{\mu}(x) \cap B_r(x)\big) \setminus \{x\} \ne \emptyset 
    	\qquad \text{for all $0 < \mu < \rho,$}
\]
and
\[
    \big(\partial\F{\rho}(x) \cap \partial\F{\mu}(x) \cap B_r(x)\big) \setminus \{x\} \ne \emptyset 
    	\qquad \text{for all $0 < \mu < \rho.$}
\]
\end{theorem}
\begin{proof}
Given $r > 0,$ choose $0 < r^\star \le r$ so that $B_{r^\star}(x) \subset \xfree$ 
and set $x^\star := x + \frac{r^\star}{\|\xi\|}\xi.$
Applying \Cref{prop:HPosNhood,lem:STLNR}, we fix $R >0$ and $\rho^\star(x) > 0$.
Choose any $\rho \in (0,\rho^\star(x))$ sufficiently small that 
$\B{\rho}(x) \subseteq B_R(x).$
    
We claim that 
\begin{equation}\label{Eqn:ReverseFromYStar}
    \R{\rho}(x) \setminus \{x\} \subset \big(B_{r^\star}(x^\star)\big)^c.
\end{equation}
Indeed, given $y \in \R{\rho}(x)$ and $\pi \in \Pi(y,x)$ with $c_{\pi} < \rho,$
consider any $\tau \in [0,T_\pi].$ Since $\pi(\tau) \in B_R(x),$ we know that
$h_F(\pi(\tau),x^\star-\pi(\tau)) > 0.$ 
It follows from \Cref{prop:DirectPropagation}
that $\|x^\star - \pi(\tau+t)\|$ is strictly decreasing 
on some interval $t \in [0,t^\star(\tau)].$ By compactness of the interval $[0,T_\pi],$
we conclude $\|x^\star - \pi(t)\|$ is strictly decreasing from $y$ to $x.$
Therefore, we have $\|x^\star - y\| > \|x^\star -x\| = r^\star.$
    
Let $0 < \mu < \rho.$ Fix $r(\mu) = r(\mu,x)$ as in
\Cref{lem:STLNR} and then select 
\[
    \tilde{r} := \min\{r,r(\mu)\}.
\]
Consider the deleted neighborhood 
$\mathcal{N} := B_{\tilde{r}}(x) \setminus \{x\}.$ 
Note that $\mathcal{N} \cap \R{\rho}(x) \ne \emptyset,$ while 
$\mathcal{N} \cap \big(\R{\rho}(x)\big)^c \ne \emptyset$ follows from
\Cref{Eqn:ReverseFromYStar}. 
Since $\mathcal{N}$ is a connected set, we conclude
\[
    \mathcal{N} \cap \partial \R{\rho}(x) \ne \emptyset.
\]
    
Let $z \in \mathcal{N} \cap \partial \R{\rho}(x).$ 
\Cref{prop:CostProperties}(e) and $\rho < \rho^\star(x)$ 
imply that $d_c(z,x) < \rho^\star(x).$
Moreover, since $\|z-x\| < \tilde{r} \le r(\mu)$, we 
conclude $d_c(z,x) < \mu,$ from \Cref{lem:STLNR}.
Therefore, we have that $z \in \R{\mu}(x) \cap \partial \R{\rho}(x)$
which implies
\[
    z \in \partial \R{\mu}(x) \cap \partial \R{\rho}(x) \cap \mathcal{N} \subseteq 
    \big(\partial\R{\rho}(x) \cap \partial\R{\mu}(x) \cap B_r(x)\big) \setminus \{x\}.
\]    
A symmetric argument proves the result for forward 
reachable sets. Therein, 
\cref{Eqn:ReverseFromYStar} is replaced by
$\F{\rho}(x)\setminus\{x\} \subset B_{r^\star}(x^\star),$
which follows directly from \Cref{prop:DirectPropagation}.
\end{proof}


\section{Intrinsic Properties of Clearance}\label{sec:CLRProps}

Having established preliminary results regarding cost balls,
we turn now to expand on the behavior of the clearance function $\clr.$ 
In particular, we focus on properties of $\clr$ observable as one traverses 
along admissible trajectories (i.e. an intrinsic perspective to objects 
moving in the system).

\begin{proposition}\label{prop:ClearanceDistanceBound}
For all $x, z \in \xspace,$ we have $\clr(x) \le \clr(z) + d_c(x,z).$
\end{proposition}
\begin{proof}
This is a straightforward result from an extended trajectory. 
Namely, suppose we have $\pi \in \Pi(x,z), \ y \in \wit(z),$ and $\pi' \in \Pi(z,y)$
so that $c_\pi = d_c(x,z)$ and $c_{\pi'} = \clr(z).$ 
It follows that the trajectory
\[
    \hat{\pi}(t) := \begin{cases} \pi(t) & \text{for $t \in [0,T_\pi]$}\\
            \pi'(t-T_\pi) & \text{for $t \in [T_\pi, T_\pi + T_{\pi'}]$}
            \end{cases}
\]
is an element of $\Pi(x,y).$
Thus, we compute
\[
    \clr(x) \le c_{\hat{\pi}} = c_{\pi'} + c_\pi = \clr(z) + 
        d_c(x,z). \qedhere
\]
\end{proof}

\begin{remark}\label{rem:NotLips}
Extending \Cref{prop:ClearanceDistanceBound},
we observe that 
\begin{equation}\label{eqn:Corollary}
    \clr(x) - \clr(z) \le d_c(x,z) \qquad \text{whenever} 
    \qquad \clr(z) < \infty.
\end{equation}
Intuition from geometric settings may lead one to 
expect the difference $\clr(z) - \clr(x)$ to also be bounded
above by $d_c(x,z).$
We demonstrate the fallacy of this attempt with the following example.
\end{remark}

\begin{ex}[Galaga System\footnote{Inspired by the classic
1981 arcade game of the same name and its subsequent
incarnations.}]\label{ex:GalagaIntro}
Working in $\xspace = \mathbb{R}^2,$ we consider the system
\begin{equation}\label{eqn:GalagaSystem}
    \begin{cases} \dot{x}_1 = u \\ \dot{x}_2 = 1 \end{cases}
    \qquad \text{for} \qquad u \in \uspace := [-1,1].
\end{equation}
For simplicity, we set $\psi \equiv 1,$ so that $c_\pi = T_\pi$ 
for all trajectories.
We constrain $\xfree$ to be a vertical 
passage opening into a wider
passage at $x_2 = 0.$ Precisely, we set
\begin{align*}
    \xfree &:= (-1,2) \times (-\infty,0] \ \bigcup \ 
    (-5,2) \times (0,\infty) \qquad \text{and} \\
    \xobst &:= \xspace\setminus \xfree.
\end{align*}
Consider $x = \left(-\frac{1}{2},-1\right)$
and $z = \left(-\frac{1}{2},0\right).$
Then, one computes (c.f.
\Cref{fig:GalagaAbruptPassage})
\begin{itemize}
    \item $d_c(x,z) = 1,$
        realized by 
        $\pi(t) = \left(-\frac{1}{2},-1+t\right)$
        with controls $u(t) \equiv 0,$
    \item $\clr(x) = 1/2,$ 
        with $\wit(x) = \left\{\left(-1,-\frac{1}{2}\right)\right\},$ and
    \item $\clr(z) = 5/2,$ 
        with $\wit(z) = \left\{\left(2,\frac{5}{2}\right)\right\}$.
\end{itemize}
Thus confirming the discussion in \Cref{rem:NotLips}
above, observing here that we have
\begin{equation}\label{eqn:NotLipschitz}
\clr(z) - \clr(x) > d_c(x,z).
\end{equation}
We also note that a discontinuous jump in $\clr$ occurs
in \Cref{fig:GalagaAbruptPassage}
when the trajectory $\pi$ passes through the point
$\left(-\frac{1}{2},-\frac{1}{2}\right).$ In fact,
we confirm in 
\Cref{cor:JumpDiscontinuousOnTrajectories}
that this is the only type of discontinuity 
that may occur along admissible trajectories.

\begin{figure}[tbhp]
\centering
\subfloat[Corner Passage]{
\includegraphics[clip=true,trim=0in 2.5in 0in 0in,height=2.25in]{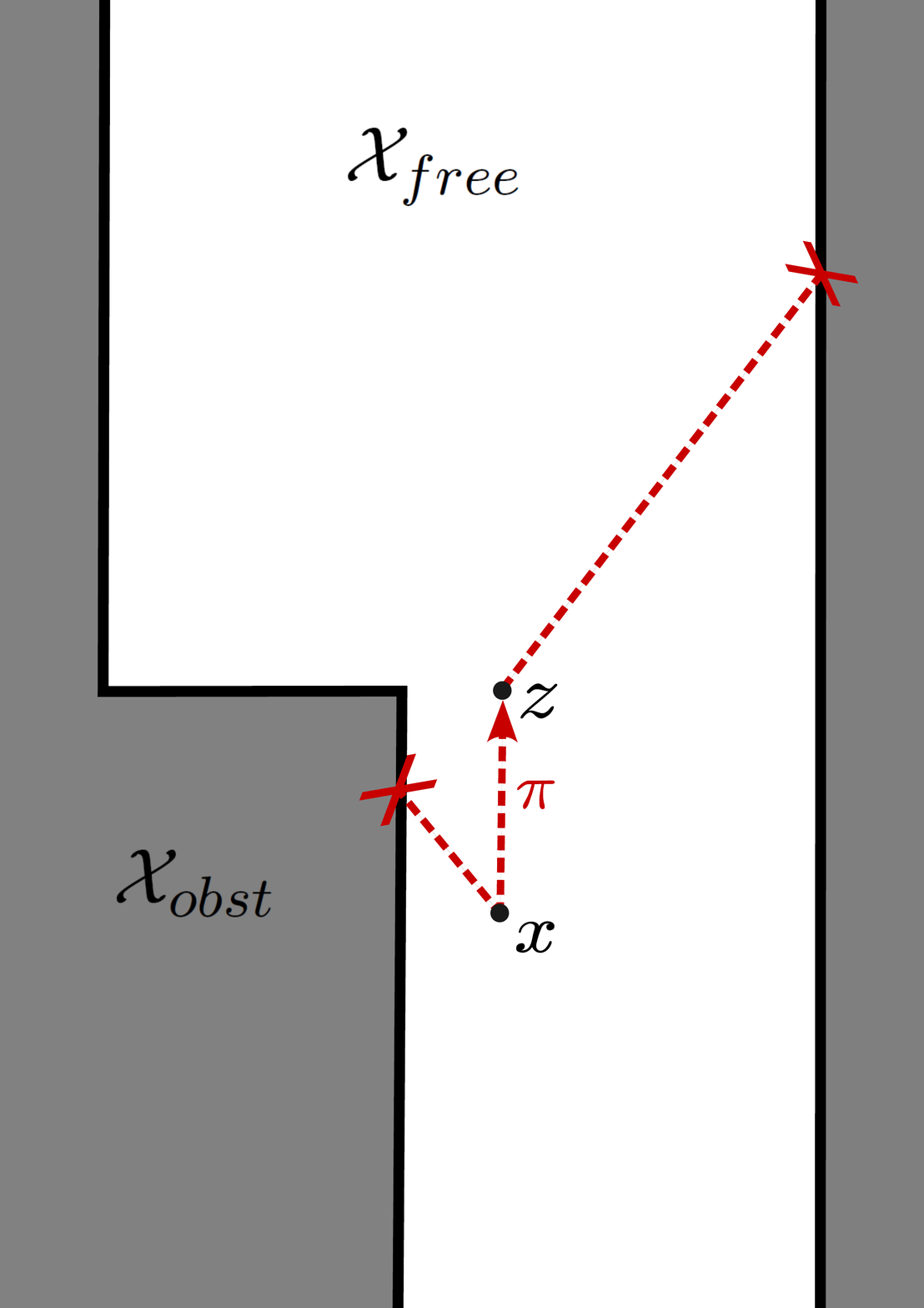}\label{fig:GalagaAbruptPassage}}
\hspace{2cm}
\subfloat[Slanted Wall]{
\includegraphics[clip=true,trim=0in 2.5in 0in 0in,height=2.25in]{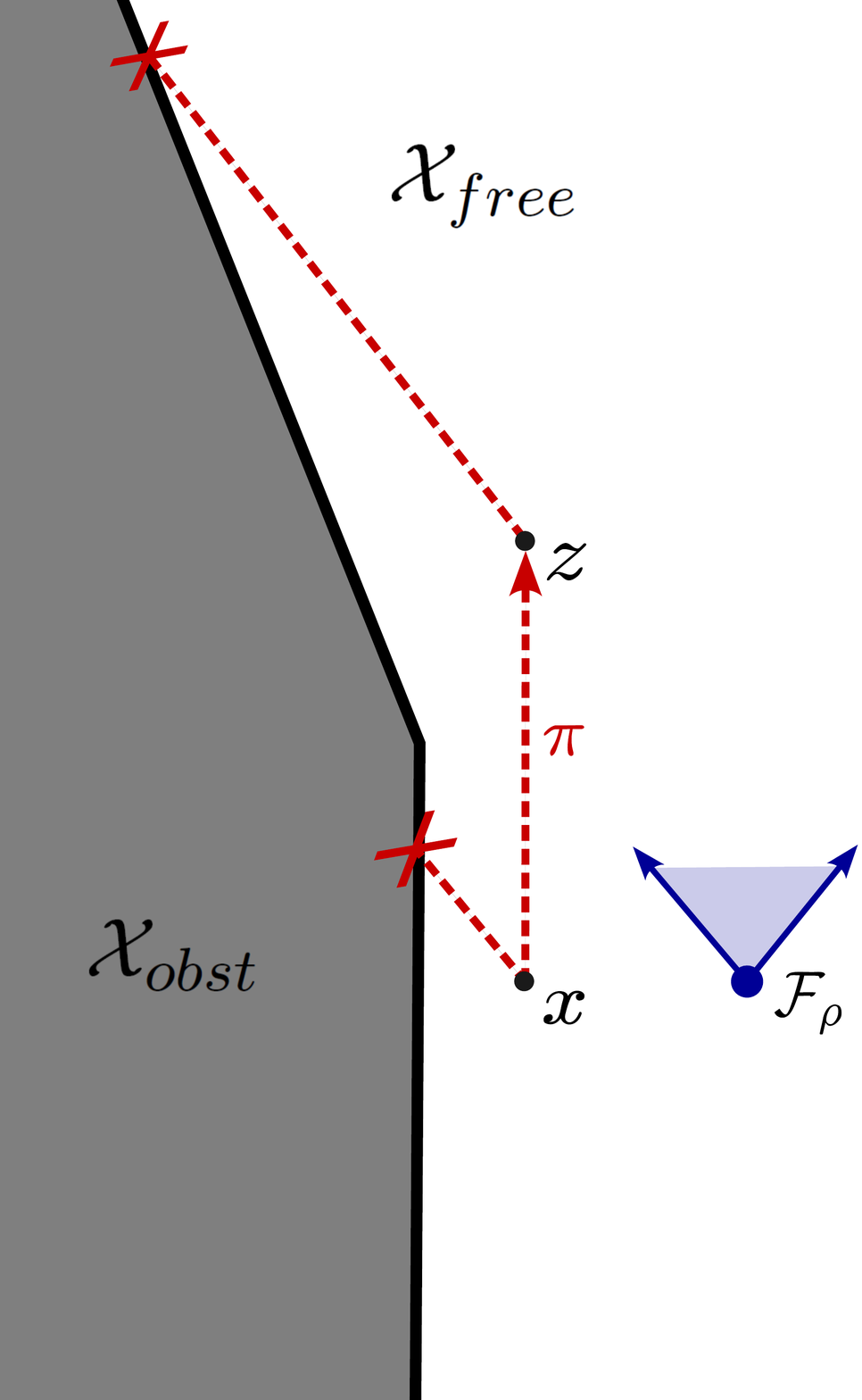}\label{fig:GalagaSlopedObstacle}}
\caption{Two $\xobst$ configurations in the Galaga model. 
Optimal trajectories (dashed lines) displayed 
propagating $x$ to $z$ and propagating 
each point to respective witness points.
A sample forward reachable set in this system
is provided (bottom right) for reference.}
\end{figure}

Further, modifying the structure of $\xfree,$ 
we can produce \Cref{eqn:NotLipschitz} 
without passing through a discontinuity in $\clr(\pi(\cdot)).$
Set
\[
    \xfree = (-1,\infty) \times (-\infty,0] \ \bigcup \ 
        \{(x_1,x_2): x_2 > -2x_1 - 2 \text{ and } x_2 > 0\},
\]
noting that $\xfree$ is no longer a long passageway,
but rather an unbounded region with only one wall bounding
motion in the negative $x_1$ direction.

Now, we consider $x = \left(-\frac{1}{2},-1\right)$
and $z = \left(-\frac{1}{2},2\right).$
Then, we have $d_c(x,z) = 3,$ while $\clr(z) = 5$
and $\clr(x) = \frac{1}{2},$ which satisfies
\Cref{eqn:NotLipschitz}.
We leave it as an exercise for the reader to confirm $\clr$ is
continuous everywhere in this modified setting.
\end{ex}

The following result establishes the (one--sided)  
limits of $\clr$ at points along admissible trajectories.

\begin{theorem}\label{thm:ClrIncreaseAlongTrajectories}
Given an admissible trajectory $\pi \in \Pi(\cdot,\cdot)$ and 
$\tau \in (0,T_\pi)$ for which $\clr(\pi(\tau))< \infty,$
it holds that
\[
    \lim_{t\to\tau^-} \ \clr(\pi(t)) \le \clr(\pi(\tau)) \le
    \lim_{t\to\tau^+} \ \clr(\pi(t)).
\]
\end{theorem}
\begin{proof}
We will first establish the existence of the one--sided limits at $\tau.$

Regarding the limit $t\nearrow\tau,$ observe that $\clr(\pi(t)) < \infty$ for all $t \in [0,\tau]$ by
\Cref{prop:ClearanceDistanceBound} and assumption that 
$\clr(\pi(\tau)) < \infty.$
Let $\varepsilon > 0,$ and fix $0< \delta$ with
$\int_{\tau-\delta}^\tau \psi(\pi(t),\dot{\pi}(t)) dt < \frac{\varepsilon}{2}.$
Next, we fix $t^\star \in (\tau-\delta,\tau)$ so that
\[
    \clr(\pi(t^\star)) > \sup_{t\in(\tau-\delta,\tau)} \clr(\pi(t)) - \frac{\varepsilon}{2}.
\]
Applying \Cref{prop:ClearanceDistanceBound}, we compute
\begin{align*}
    \liminf_{t\to\tau^-} \ \clr(\pi(t)) &\ge \inf_{t\in(t^\star,\tau)} \clr(\pi(t))\\
    &\ge \inf_{t\in(t^\star,\tau)} \big(\clr(\pi(t)) + d_c(\pi(t^\star),\pi(t))\big) 
    - \int_{\tau-\delta}^\tau \psi(\pi(t),\dot{\pi}(t)) dt\\
    &> \clr(\pi(t^\star))-\frac{\varepsilon}{2}\\
    &> \sup_{t\in(\tau-\delta,\tau)}\clr(\pi(t))-\varepsilon\\ 
    &\ge \limsup_{t\to\tau^-} \ \clr(\pi(t))-\varepsilon.
\end{align*}
Taking the limit $\varepsilon \to 0,$ we conclude that $\lim_{t\to\tau^-} \clr(\pi(t))$
is well--defined.

The proof that $\lim_{t\to\tau^+}\clr(\pi(t))$ is well--defined 
follows in a symmetric manner if there exists any $\hat{t}\in(\tau,T_\pi)$
at which $\clr(\pi(\hat{t})) < \infty.$
Alternatively, if no such $\hat{t}$ exists,
the one--sided limit exists with
$\lim_{t\to\tau^+} \clr(\pi(t)) = \infty.$

Finally, the proof follows from 
\Cref{prop:ClearanceDistanceBound} again. 
In particular, we establish that
\[
    \clr(\pi(t_1)) -d_c(\pi(t_1),\pi(\tau)) \le \clr(\pi(\tau))
    \le \clr(\pi(t_2)) + d_c(\pi(\tau),\pi(t_2)),
\]
holds for all $0 < t_1 < \tau < t_2 < T_\pi,$ and we have 
\[
    \lim_{t_1\to\tau^-}d_c(\pi(t_1),\pi(\tau)) = 
    \lim_{t_2\to\tau^+}d_c(\pi(\tau),\pi(t_2)) = 0. \qedhere
\]
\end{proof}


\begin{lemma}\label{lem:ClrToWitness}(Principle of Optimality)
If $\clr(x) < \infty, \ y \in \wit(x),$ and $\pi \in \Pi(x,y)$ is an optimal trajectory, then
\[
    \clr(x) - \clr(\pi(t)) = d_c(x,\pi(t)) = c_{\pi(0,t)} \qquad \text{for all $t \in [0,T_\pi].$}
\]

\end{lemma}
\begin{proof}
    This is a standard result in optimal control theory with important
    connections to the theory of Hamilton--Jacobi equations and
    clearance / minimal time functions.
    For the reader's convenience, we provide an elementary proof
    
    Choose $t \in [0,T_\pi]$ and define $\pi(0,t) := \pi\big|_{[0,t]} \in \Pi(x,\pi(t)),$
    the restricted trajectory.
    By Pontryagin's Maximum Principle,
    we know that every segment of an optimal
    trajectory is likewise optimal. Thus, we know that $d_c(x,\pi(t)) = c_{\pi(0,t)}$
    and so 
    \[
        \clr(\pi(t)) \le c_\pi - c_{\pi(0,t)} = \clr(x) - d_c(x,\pi(t)).
    \]
    We thus have $\clr(\pi(t)) \le \clr(x) < \infty.$ 
    Meanwhile, applying \Cref{eqn:Corollary}, we derive	
    \[
        \clr(\pi(t)) \ge \clr(x) - d_c(x,\pi(t)),
    \]
    which concludes the proof.
\end{proof}

\section{Discontinuities of Clearance Functions}\label{sec:CLRDisc}

To study discontinuities of $\clr$, we introduce a structure modeling the
uniform propagating waves in $\xfree$ of points with
increasing clearance from $\xobst.$
This is akin to a solution to the eikonal equation (or 
grassfire algorithm) with an added restriction
that moving surfaces propagate only backward along 
admissible trajectories (c.f. \cite{N07} for a detailed discussion
of connections between minimal time functions and eikonal equations).
Quasi--stationary wave boundaries (or envelopes) can be observed in 
certain control systems, when configurations of $\xobst$ interact with
constraints on local controllability.
We investigate the fine properties of these envelopes in this section.

\begin{definition}\label{Def:WaveEnvelope}
Given $\rho > 0$, define the {\bf propagating wave} 
\[
    W_\rho := \bigcup_{y \in \xobddry} \R{\rho}(y) = \{ x \in \xfree : \clr(x) < \rho\}.
\]
For every point $x \in \xfree,$ denote the first and last 
arrival of wave fronts as
\[
    \rho_{min}(x) := \inf \{ \rho > 0 : x \in \partial W_\rho \}
    \quad \text{and} \quad
    \rho_{max}(x) := \sup \{ \rho > 0 : x \in \partial W_\rho \}.
\]
By convention, we set $\rho_{min}(x) = \rho_{max}(x) = \infty$ 
when $\clr(x) = \infty.$
Finally, define the {\bf wave envelope} of all states 
where $\partial W_\rho$ persists for some nontrivial time; 
\[
    E := \{ x\in\xfree:\rho_{min}(x)<\rho_{max}(x) \}.
\]
\end{definition}

\begin{remark}\label{rem:Envelope} \mbox{ }
\begin{enumerate}
    \item If $x \in \xfree \setminus E,$ then 
        $\rho_{min}(x) = \rho_{max}(x) = \clr(x).$
        Moreover, $x \in (W_\rho^c)^\circ$ for all $\rho < \clr(x)$
        and $x \in (W_\rho)^\circ$ for all $\rho > \clr(x).$
    \item Given $x \in E,$ it holds that $x \in \partial W_\rho$
        for all $\rho_{min}(x) \le \rho < \rho_{max}(x).$
    \item By lower semi--continuity of $\clr$ (c.f. \cite[Proposition 2.6]{WZ98}),
       we know that $\rho_{min}(\cdot) = \clr(\cdot)$ on $\xspace.$
\end{enumerate}
\end{remark}

\subsection{Discontinuities of $\clr$ in $\xfree$}
The following result characterizes all discontinuities in
clearance that arise away from the obstacle.

\begin{theorem}\label{thm:DiscontinuousOnE}
Let $x \in \xfree$. Then $\clr$ is discontinuous at $x$
if and only if $x \in E.$
\end{theorem}

\begin{proof}
Addressing the necessity statement, assume $x \in E.$ 
Fix $\rho_{min}(x)<\nu<\mu<\rho_{max}(x),$
so that $x \in \partial W_\nu \cap \partial W_\mu.$
Then, there exist two sequences, $(y_m)$ and $(z_m),$ converging to $x,$
with $y_m \in \xfree \setminus W_{\mu}$ and 
$z_m \in W_{\nu}.$
Thus, we have $\clr(y_m) \ge \mu > \nu > \clr(z_m)$ for all $m$.
So, $\clr$ is discontinuous at $x$.

Addressing the sufficiency statement, assume $x \in E^c$
and let $(y_m) \subset \xfree$ be any sequence converging to $x.$
For any $\rho < \clr(x) = \rho_{min}(x),$ it holds that 
$x \in (W_\rho^c)^\circ.$ 
It follows that $y_m \in W_{\rho}^c,$ and consequently 
$\clr(y_m) \ge \rho,$ for $m$ sufficiently large.
Taking $\rho \nearrow \clr(x)$, we conclude
\begin{equation}\label{liminfBound}
    \liminf_{m \to \infty} \clr(y_m) \ge \clr(x).
\end{equation}
Likewise, for any $\rho > \clr(x) = \rho_{max}(x),$
it holds that $x \in (W_\rho)^\circ.$
It follows that $y_m \in W_{\rho},$ and consequently 
$\clr(y_m) < \rho,$ for $m$ sufficiently large.
Taking $\rho \searrow \clr(x)$, we derive
\[
    \limsup_{m \to \infty} \clr(y_m) \le \clr(x).
\]
Together with \Cref{liminfBound}, this proves that $\clr$ is continuous at $x.$
\end{proof}

We now confirm that, while traversing admissible trajectories 
$\pi$ in $\xfree,$ all discontinuities in $\clr(\pi(\cdot))$ 
must be accompanied by an instantaneous increase in clearance
(c.f. the discussion in \Cref{ex:GalagaIntro}).

\begin{corollary}\label{cor:JumpDiscontinuousOnTrajectories}
Suppose $\pi \in \Pi(x,y)$ is an admissible trajectory 
for which $\clr(\pi(\cdot))$ is discontinuous at some 
$\tau \in (0,T_\pi).$ Then 
\[
    \lim_{t\to\tau^-} \clr(\pi(t)) < \lim_{t\to\tau^+} \clr(\pi(t)).
\]
Further, it holds that $\pi(\tau) \in E,$ 
provided $\pi(\tau) \in \xfree.$
\end{corollary}
\begin{proof}
This result follows directly from
\Cref{thm:ClrIncreaseAlongTrajectories,thm:DiscontinuousOnE}.
\end{proof}

Turning from trajectories that pass through the wave envelope $E,$
the next result considers trajectories that travel along the envelope.
This result gives further insight into the structure of $E$ and how 
it connects to $\xobst.$

\begin{theorem}\label{thm:EAlongTrajectories}
Given $x \in E.$
If $y \in \wit(x)$ and $\pi \in \Pi(x,y)$ is an optimal trajectory,
then $\pi(t) \in E$ for all $t \in [0,T_\pi).$
\end{theorem}
\begin{proof}
Set $0 < \varepsilon < \frac{1}{2}(\rho_{max}(x) - \clr(x))$ and 
$t \in [0,T_\pi).$ Let $r > 0$ be arbitrary.

By \Cref{prop:CostProperties}(d), there exists 
$\delta > 0$ so that for all $z \in B_{\delta}(x),$ there exists 
$\pi' \in \Pi(z,B_r(\pi(t)))$ transporting $z$ into $B_r(\pi(t)),$ with
\[
    c_{\pi'} = c_{\pi(0,t)} = \int_0^t \psi(\pi(s),\dot{\pi}(s)) ds.
\]
Here we recall that $\pi(0,t) := \pi\big|_{[0,t]} \in \Pi(x,\pi(t))$ 
denotes the restricted trajectory.

Let $\rho \in (\clr(x)+\varepsilon, \rho_{max}(x)).$
By \Cref{rem:Envelope}(b,c) we know that $x \in \partial W_\rho.$
Thus, we can select $z \in B_\delta(x) \cap W_\rho^c$ and compute;
applying \Cref{prop:ClearanceDistanceBound} along $\pi'$
and then \Cref{lem:ClrToWitness} along $\pi(0,t),$
\begin{align*}
  \clr(\pi'(T_{\pi'})) &\ge \clr(z) - d_c(z,\pi'(T_{\pi'}))
    \ge \clr(z) - c_{\pi'}\\
    &> (\clr(x) + \varepsilon) - c_{\pi(0,t)} = 
    \clr(\pi(t)) + \varepsilon. 
\end{align*}
Since $\pi'(T_{\pi'}) \in B_r(\pi(t))$ and $r > 0$ is arbitrary, we 
conclude that $\clr$ is discontinuous at $\pi(t) \in \xfree.$
The claim follows by \Cref{thm:DiscontinuousOnE}.
\end{proof}

\subsection{Discontinuities of $\clr$ on $\xobddry$}

Having characterized all discontinuities of $\clr$ in $\xfree$ 
(\Cref{thm:DiscontinuousOnE}),
we turn now to discuss discontinuities on $\xobddry.$ 
The analysis on the boundary differs, by necessity, from our
approach in $\xfree.$
Notably, observe that many points $y \in \xobddry$ 
will generically reside in $\partial W_\rho$ for all $\rho > 0,$
though this persistence of wave boundaries is certainly not an 
indication of discontinuity.
Indeed, if $\clr$ is continuous at $y \in \xobddry,$
then it must hold that $y \in \partial W_\rho$ for all $\rho > 0.$

We begin with two propositions demonstrating how
interactions between reachable sets (centered at $y \in \xobddry$)
and $\xobst$ lead to clearance discontinuities.
Informally, the first proposition derives discontinuity of 
$\clr$ at $y \in \xobddry$ if all obstacle boundary points 
in a neighborhood are uniformly unreachable from free space.
The second proposition is a small modification of the first,
assuming that free space is instantaneously accessible from $y,$ 
while all boundary points that happen to be in a forward reachable
set from $y$ are uniformly unreachable from free space.
We note that neither of these results requires any explicit 
directionality in the admissible velocities at $y,$
though such a condition may be a consequence of our so--called
uniform unreachability assumption. 

\begin{proposition}\label{prop:CliffInteriorDisc}
Suppose $y_0 \in \xobddry$ and there exist $r, \rho > 0$ so that 
$\R{\rho}(y) \cap \xfree = \emptyset$
for all $y \in B_r(y_0) \cap \xobddry$. Then $\clr$ is discontinuous at $y_0.$
\end{proposition}

\begin{proof}
By compactness of $\overline{B_{r/2}(y_0)},$ we fix $\tilde{\rho} > 0$ sufficiently
small that $\B{\tilde{\rho}}(x) \subset B_{r/2}(x)$ for all $x \in B_{r/2}(y_0).$
For all such $x$ it either holds that $\wit(x) \cap B_r(y_0) \ne \emptyset,$
in which case $\clr(x) \ge \rho,$ or $\wit(x) \cap B_r(y_0) = \emptyset,$
in which case we conclude $\clr(x) \ge \tilde{\rho}.$
Thus, there is a type of jump discontinuity in clearance at $y_0,$ in
the sense that given any sequence $(x_n) \subset \xfree$ converging to $y_0,$
it follows that
\[
    \liminf_{n \to \infty} \clr(x_n) \ge \min\{\rho, \tilde{\rho}\} > 0. \qedhere
\]
\end{proof}

\begin{proposition}\label{prop:CliffForwardDisc}
Suppose $y_0 \in \xobddry$ with $\F{\mu}(y_0) \cap \xfree \ne \emptyset$ for all $\mu > 0,$
and there exist $\rho, \rho_0 > 0$ so that, for all $y \in \F{\rho_0}(y_0) \cap \xobddry,$
it holds that $\R{\rho}(y) \cap \xfree = \emptyset.$
Then $\clr$ is discontinuous at $y_0.$
\end{proposition}

\begin{proof}
This argument is essentially the same as \Cref{prop:CliffInteriorDisc},
with the modification that we consider $x \in \F{\rho_0/2}(y_0)$ and note
that $\F{\eta}(x) \subset \F{\rho_0}(y_0)$ for all $\eta \in [0,\rho_0/2].$
We thus conclude that either $\clr(x) \ge \rho_0/2$ 
(whenever $\wit(x) \cap \F{\rho_0}(y_0) = \emptyset$),
or else $\clr(x) \ge \rho$ (when $\wit(x) \cap \F{\rho_0}(y_0) \ne \emptyset$).
\end{proof}

We close out the paper focusing on a particular type of 
discontinuous boundary point that appears to generate 
free space discontinuities.
That is, we study the points $y_0 \in \xobddry$ around which one can find
envelope points propagating out into $\xfree$ on the boundaries 
of arbitrarily small waves.
We introduce the following precise definition.

\begin{definition}\label{def:EnvGen}
An obstacle boundary point $y_0 \in \xobddry$ is called an {\bf envelope generator}
if there exists a sequence $x_n \in E$ with $x_n \rightarrow y_0$
and $\rho_{min}(x_n) \rightarrow 0.$
\end{definition}

\begin{ex} We first present a few simple settings where
envelope generators are readily identifiable by way of \Cref{thm:EAlongTrajectories}.
\begin{enumerate}
    \item (Galaga: Sharp Corner) Consider the system
    \Cref{eqn:GalagaSystem} with the abrupt passageway configuration 
    for $\xobst$ (\Cref{fig:GalagaAbruptPassage}).
    It holds in this setting that $y_0 = (-1,0)$ is an 
    envelope generator.
    To see this, we apply \Cref{thm:EAlongTrajectories},
    noting, for instance, that 
    $x = \left(-\frac{1}{2},-\frac{1}{2}\right) \in E$
    and $y_0 \in \wit(x),$ with $\pi(t) = x + (-t,t)$
    an optimal trajectory connecting $x$ to $y_0.$  
    
    \item (Generalized Galaga: Sharp Corner)
    Generalizing the Galaga system, we allow for control 
    of acceleration in the $x_2$--direction. Namely, consider
\begin{equation}\label{eqn:ModGalagaSystem}
    \begin{cases} \dot{x}_1 = u_1 \\ 
        \dot{x}_2 = x_3 \\
        \dot{x}_3 = u_2
    \end{cases}
    \qquad \text{for} \qquad (u_1,u_2) \in \uspace := [-1,1]^2.
\end{equation}
    We work in $\xspace = \mathbb{R}^3$ and consider the following
    extension of our obstacles,
\begin{align*}
    \xfree &= (-1,2) \times (-\infty,0] \times \mathbb{R} 
        \ \bigcup \ (-5,2) \times (0,\infty) \times \mathbb{R} 
        \qquad \text{and} \\
    \xobst &= \xspace\setminus \xfree.
\end{align*}
    For simplicity, we again consider $\psi \equiv 1,$ though the following
    argument is essentially unchanged under a number of simple running cost
    functions. For instance, one might consider
    $\psi(\pi, \dot{\pi}) := \sqrt{\|(\dot{\pi}_1, \dot{\pi}_2, 0)\|^2},$
    so that $c_\pi$ is the arclength of the path described by the trajectory $\pi$
    as seen in the $x_1x_2$--plane.
    
    Given any $v > 0,$ it follows from
    \Cref{thm:EAlongTrajectories} that $y_0 := (-1,0,v)$ 
    is an envelope generator. To see this, we note that
    \[
        x = \left(-\frac{1}{2}, 
            -\frac{1}{4}\big(1+2v\big),v+\frac{1}{2}\right) \in \xfree
    \]
    and $y_0 \in \wit(x),$ with optimal trajectory
    \[
        \pi(t) = x + \left(-t,vt+\frac{t}{2}-\frac{t^2}{2},-t\right)
    \]
    connecting $x$ to $y_0.$ To see that $x \in E,$ we note
    that $\clr(x) = d_c(x,y_0) = \frac{1}{2},$ while the clearance of 
    any point 
    \[
        x_n := x + \left(\frac{1}{n}, 0,0\right) \qquad 
            \text{for $n$ sufficiently large}
    \]
    is determined by its cost distance to witness points on the right wall;
    $y = (2,\cdot,\cdot) \in \xobddry$ (\Cref{fig:GalagaModLrg,fig:GalagaModMed}) 
    or on the shelf above the narrow passage;  
    $y' = (\cdot,0,\cdot) \in \xobddry$ (\Cref{fig:GalagaModSmall}),
    since none of these points can propagate to the left wall (prior to the 
    corner) along admissible trajectories.
    We thus conclude that $\limsup_{z\to x}\clr(z) > \clr(x) = \frac{1}{2}.$
\end{enumerate}
\begin{figure}[tbhp]
\centering
\subfloat[$x_3 >\!> 0$]{
\includegraphics[clip=true,trim=0.5in 1.5in 0.5in 2in,height=1.85in]{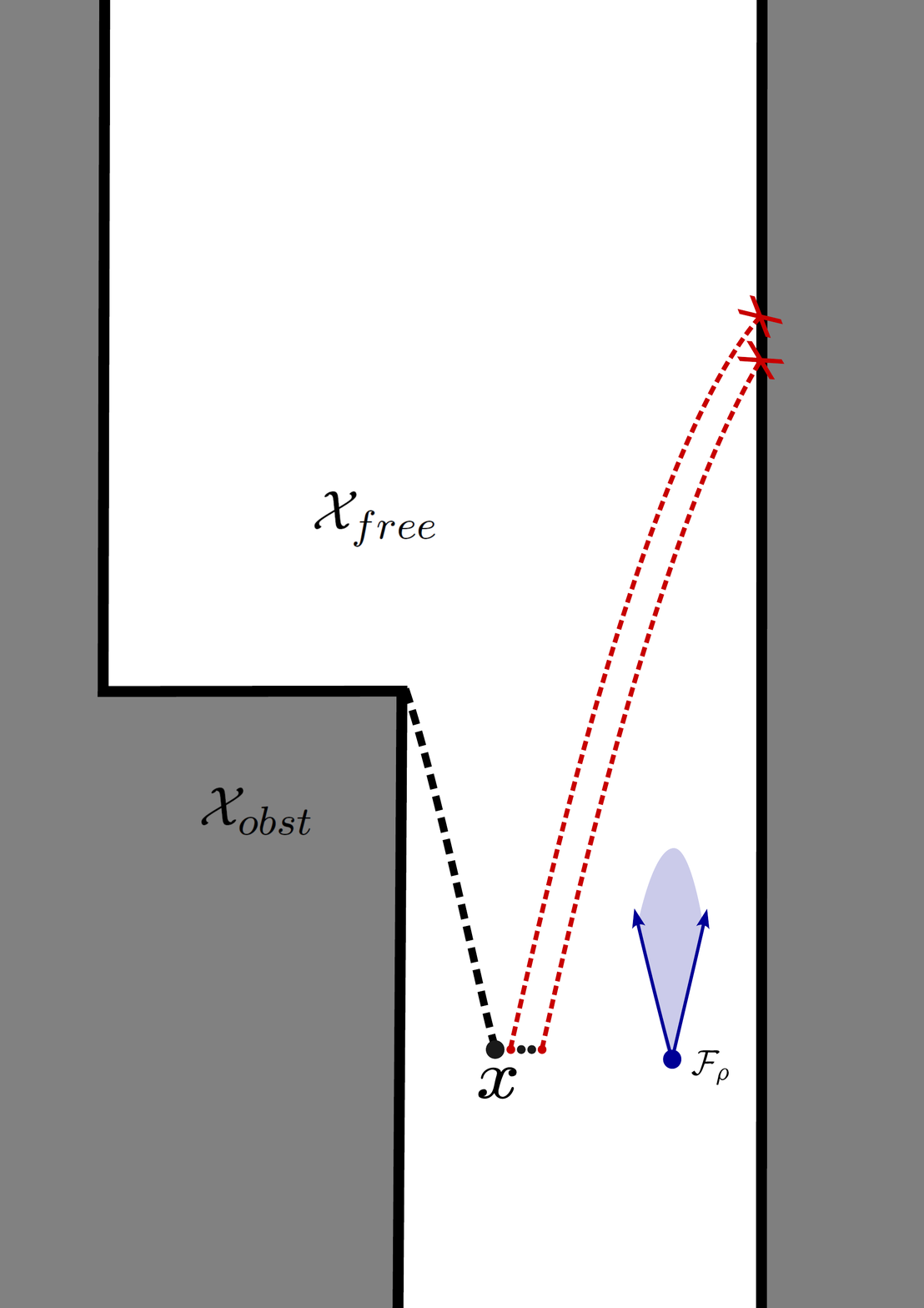}\label{fig:GalagaModLrg}}
\hspace{0.15cm}
\subfloat[$x_3 > 0$]{
\includegraphics[clip=true,trim=0.5in 1.5in 0.5in 2in,height=1.85in]{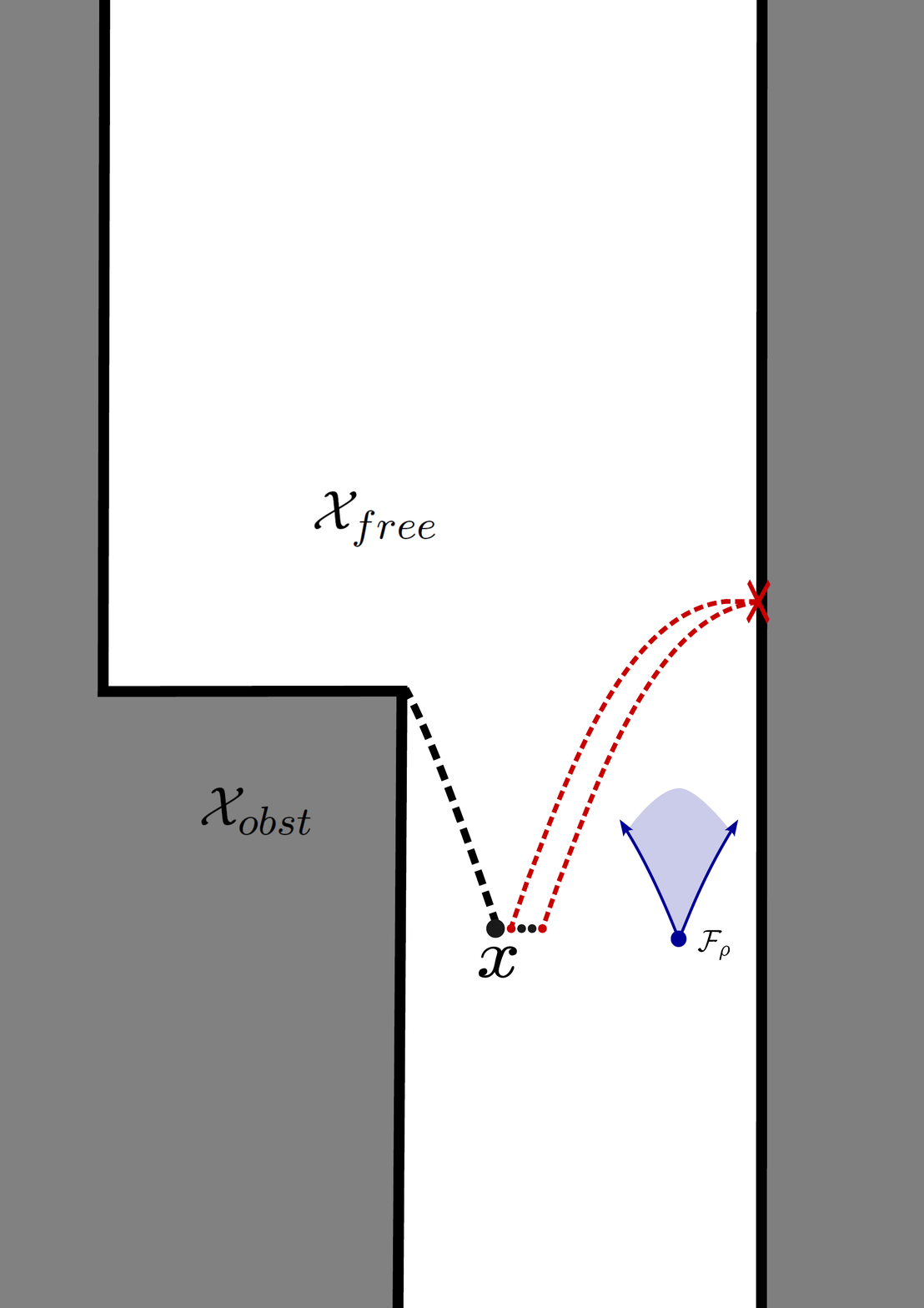}\label{fig:GalagaModMed}}
\hspace{0.15cm}
\subfloat[$x_3 \approx 0$]{
\includegraphics[clip=true,trim=0.5in 1.5in 0.5in 2in,height=1.85in]{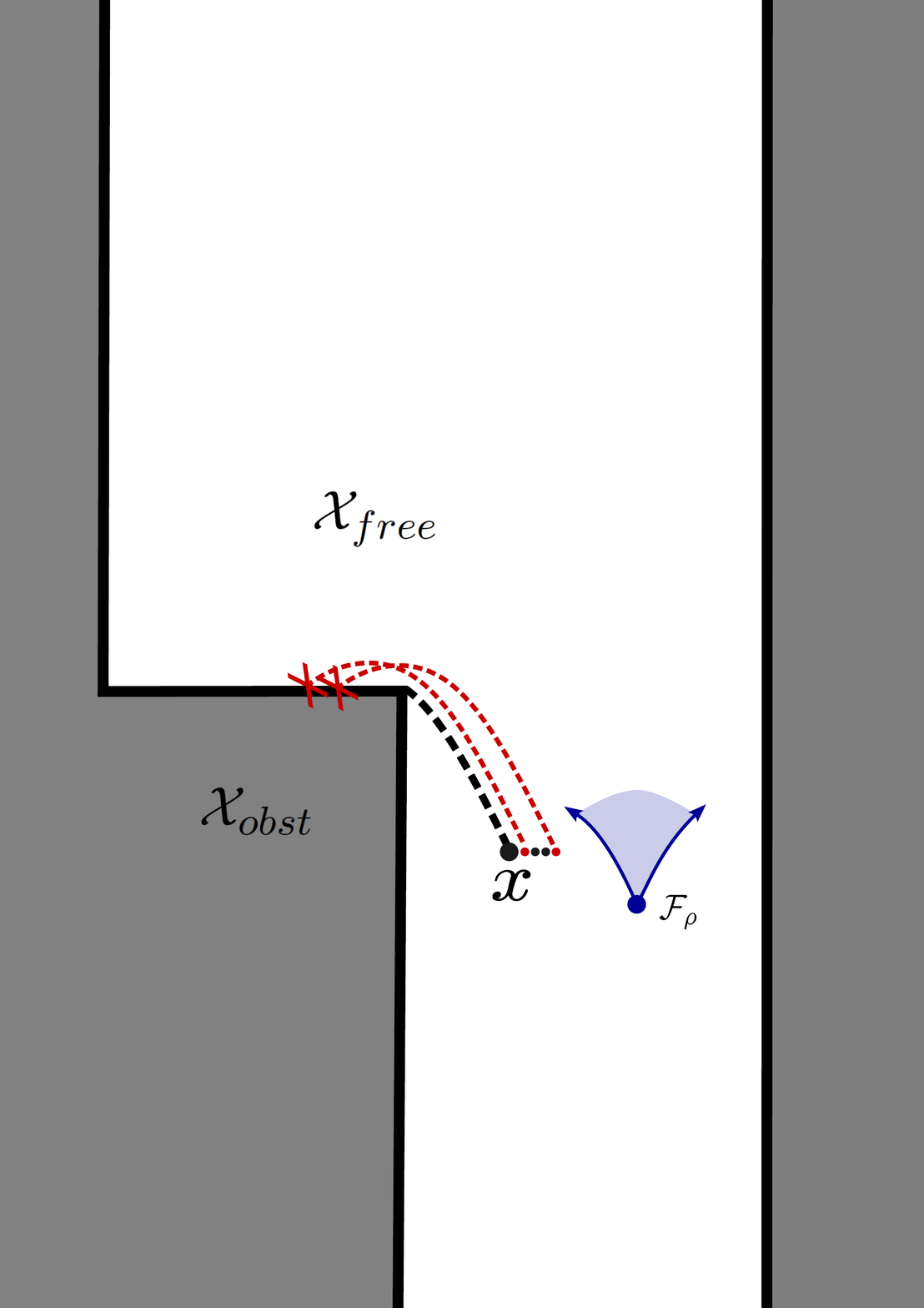}\label{fig:GalagaModSmall}}
\caption{Selections of points $x$ in the Generalized Galaga system with 
clearance witnessed by $y_0 := (-1,0,v),$ accompanied by points nearby
confirming the fact that $x \in E.$
All images are projections of $\xspace$ onto the 
$x_1x_2$--plane (the configuration space here).}
\label{fig:ModifiedGalaga}
\end{figure}
\end{ex}

\subsection{Sufficient Conditions for Envelope Generators}

We conclude the paper with our main theorem, providing sufficient conditions
that guarantee a selected point $y_0 \in \xobddry$ is an envelope generator.
For convenience of notation, we introduce the following sets which decompose 
$\xobddry.$
Motivated by the shelf points in the Galaga systems above 
(i.e. $(x_1,x_2) \in \xobddry$ with $x_2 = 0$), we define
\begin{equation}\label{eqn:ShelfDef}
    \mathcal{S} := \left\{ y \in \xobddry: \R{\rho}(y) \cap \xfree = \emptyset 
        \ \text{for some $\rho > 0$} \right\},
\end{equation}
while the cliff points (i.e. $(x_1,x_2) \in \xobddry$ with $x_1 = -1$) 
motivate the definition
\begin{equation}\label{eqn:CliffDef}
    \mathcal{C} := \left\{ y \in \xobddry: \R{\rho}(y) \cap \xfree \ne \emptyset 
        \ \text{for all $\rho > 0$} \right\}.
\end{equation}
Additionally, we introduce the following assumptions on the structure
of the control system and the obstacle set, in a neighborhood of 
$y_0 \in \xobddry.$ 

\begin{itemize}
\item[{\bf (H1)}] There exists a vector 
$\xi \in \xspace$ and a radius $r^\star > 0$ so that
    \begin{itemize}
    \vspace{0.5em}
    \item[a)] $\displaystyle 
        h_F(y_0,\xi) := \min_{v \in F(y_0)} \langle v , \xi \rangle > 0,$ and 
    \item[b)] defining the ball 
    $\displaystyle B_{r^\star}(y^\star) := B_{r^\star}\left(y_0 + \frac{r^\star}{\|\xi\|} \xi \right),$
    it holds that
    \[
        B_{r^\star}(y^\star) \cap \xobddry \subseteq \mathcal{S}.
    \]
    \end{itemize}
    \item[{\bf (H2)}] Locally, the structure of $\xobst$ is such that $B_r(y_0) \cap \xfree$
        is a connected set for all $0 < r < r_0$ sufficiently small.
\end{itemize}

\begin{lemma}\label{lem:NoWitnessesInStarBall}
Suppose {\bf H1(b)} holds and $x \in \xfree.$
Then $\wit(x) \cap B_{r^\star}(y^\star) = \emptyset.$
\end{lemma}

\begin{proof}
    Let $x \in \xfree.$ 
    Suppose $y \in \xobddry \cap B_{r^\star}(y^\star)$
    and $\pi(\cdot) \in \Pi(x,y).$
    Define $\tau_\pi := \sup \{t > 0: \pi(t) \in \xfree\}.$
    Note that $\pi(\tau_\pi) \in \xobddry.$
    By {\bf H1(b),} we have $y \in \mathcal{S}$ and
    so $\tau_\pi < T_\pi.$
    Thus, we conclude that $y \notin \wit(x).$
\end{proof}

\begin{proposition}\label{prop:H1Disc}
If {\bf H1} holds and 
$\F{\rho}(y_0) \cap \xfree \ne \emptyset$ for all $\rho > 0$,
then $\clr$ is discontinuous at $y_0.$
\end{proposition}

\begin{proof}
    By assumption, there exists a maximally defined 
    admissible trajectory $\pi(\cdot) \in \Pi(y_0, \cdot)$
    with $\pi(t) \in \xfree$ for all $0 < t < t_0$ sufficiently
    small (i.e. $\pi(\cdot)$ propagates immediately 
    from $y_0$ into free space).
    By \Cref{lem:UnifPenetration}, we see that 
    $\pi(t) \in B_{r^\star}(y^\star)$ for $t > 0$ sufficiently
    small, and so it follows from 
    \Cref{lem:NoWitnessesInStarBall} that 
    $\wit(\pi(t)) \subset \big(B_{r^\star}(y^\star)\big)^c$
    for all such $t > 0.$
    Finally, note that the cost associated with any trajectory
    propagating from $\overline{B_{\eta^\star r}(y^\star)}$ to
    $(B_{r}(y^\star))^c$ is strictly positive. 
    Thus, compactness of $\overline{B_{\eta^\star r}(y^\star)}$ 
    and \Cref{lem:UnifPenetration,prop:ClearanceDistanceBound} imply
    \[
        \liminf_{t \searrow 0} \clr(\pi(t)) \ge \tilde{\rho} > 0,
            \qquad \text{for some $\tilde{\rho} > 0.$}
    \]
    This is sufficient to conclude discontinuity of $\clr$ at $y_0.$
\end{proof}

Now we are prepared to prove our sufficiency result for envelope generators.

\begin{theorem}\label{thm:EnvGenThm}
Suppose assumptions {\bf (H1)} and {\bf (H2)} hold. Additionally assume that
$\F{\rho}(y_0) \cap \xfree \ne \emptyset$ and 
$\R{\rho}(y_0) \cap \xfree \ne \emptyset,$
for all $\rho > 0.$ 
Then $y_0$ is an envelope generator.
\end{theorem}
\begin{proof}
    Let $r, \rho > 0$ arbitrary. We will show that there exists some $x \in B_r(y_0) \cap E$
    with $\rho_{min}(x) < \rho.$
    
    Using {\bf H1(a),} as in the proof of
    \Cref{lem:UnifPenetration}, 
    select $0 < R \le \min\{r,r_0\}$ sufficiently small 
    that $h_F(x,y^\star-x) > \frac{1}{2}h_F(y_0,y^\star - y_0)$
    for all $x \in \overline{B_R(y_0)}.$ 
    Now, define
    \[
        \psi^\star := \max \{\psi(x,v): x \in \overline{B_R(y_0)}, v \in F(x) \}
    \]
    and, using \Cref{lem:STLNR}, define
    \[
        \rho^\star := \min\left\{ \rho^\star(x) : x \in \overline{B_{R}(y_0)}\right\}.
    \]
    Select a value $0 < \rho_1 < \min\{\rho, \rho^\star\}$ sufficiently small that
    \[
        \B{\rho_1}(x) \subseteq B_{R/3}(x) \qquad 
        \text{for all $x \in \overline{B_{R}(y_0)}.$}
    \]
    We proceed with successive scalings
    employing \Cref{lem:STLNR}.
    For $i = 1,2,$ define
    \[
        r_i := \min\left\{r(\rho_i,x): x \in \overline{B_{R}(y_0)}\right\},
    \]
    where $\rho_{2}$ is chosen appropriately small to ensure
    \[
        \B{\rho_{2}}(x) \subseteq B_{r_{1}}(x) \qquad
        \text{for all $x \in \overline{B_{R}(y_0)}.$}
    \]
    As constructed, it holds that $\rho_1$ and $\rho_2$ 
    are both positive, with the property
    \[
        \B{\rho^\star}(x) \setminus \B{\rho_i}(x) \subset 
        \big( B_{r_i}(x)\big)^c \qquad 
        \text{for $i = 1,2$ and $x \in \overline{B_{R}(y_0)}.$}
    \]
    We also observe that $r_2 \le r_1 \le R/3.$\\
    
    \noindent{\bf Claim 1:} There exists $\tilde{r} > 0$ so that for all 
    $x \in B_{\tilde{r}}(y_0) \cap \xfree,$ it holds that either
    $\clr(x) < \rho_2$ or $\clr(x) \ge \rho_1.$
    
    \noindent Consider any radius $0 < r < R/3,$ any point 
    $x \in B_r(y_0),$ and a maximally defined admissible 
    trajectory $\pi(\cdot) \in \Pi(x,\cdot).$
    For all $t < \min\{T_\pi, \rho_1\},$ we have
    $\pi(t) \in B_{R}(y_0).$
    Thus, as in the proof of \Cref{lem:UnifPenetration}, 
    we compute 
    \[
        \frac{d}{dt}\|y^\star - \pi(t)\| \le - \ \frac{h_F(y_0,y^\star-y_0)}{r^\star},
    \]
    and so it follows that $\pi(t) \in B_{r^\star}(y^\star)$ 
    whenever 
    \[
        t > t^\star(r) := \frac{2r^\star\big((r+r^\star)^2 - (r^\star)^2\big)}{h_F(y_0,y^\star-y_0)}.
    \]
    We note that $t^\star(r) \to  0$ as $r \to 0^+,$
    so we fix $0 < \tilde{r} < R/3$ such that 
    $t^\star(\tilde{r}) \psi^\star < \rho_2.$ 
    
    Given $x \in B_{\tilde{r}}(y_0)$, consider $\pi \in \Pi(x,\cdot)$ with $c_\pi \ge \rho_2.$
    It follows that 
    \[
       T_\pi \psi^\star \ge \int_0^{T_\pi}\psi(\pi(t),\dot{\pi}(t)) dt = c_\pi \ge \rho_2.
    \]
    Thus, $T_\pi > t^\star(\tilde{r})$ and so $\pi(T_\pi) \in B_{r^\star}(y^\star).$
    We conclude that 
    \begin{equation}\label{eqn:IntoTheStarBall}
        \F{\rho_1}(x) \setminus \F{\rho_2}(x) \subseteq B_{r^\star}(y^\star)
        \qquad \text{for all $x \in B_{\tilde{r}}(y_0).$}
    \end{equation}
    The validity of {\bf Claim 1} now follows from 
    \Cref{lem:NoWitnessesInStarBall}.\\

    \noindent{\bf Claim 2:} $B_{\tilde{r}}(y_0) \cap \partial W_{\rho_1} \ne \emptyset.$
    
    \noindent First, observe that 
    \begin{equation}\label{eqn:WaveIn}
        B_{\tilde{r}}(y_0) \cap W_{\rho_1} \ne \emptyset,
    \end{equation}
    by assumption that 
    $\R{\rho}(y_0) \cap \xfree \ne \emptyset$
    for all $\rho > 0.$
    Meanwhile, by the forward accessibility assumption 
    ($\F{\rho}(y_0) \cap \xfree \ne \emptyset$ 
    for all $\rho > 0$) we select 
    $x \in \F{\rho_1}(y_0) \cap B_{\tilde{r}}(y_0) \cap \xfree.$
    By \Cref{lem:UnifPenetration}, it holds that
    $\F{\rho_1}(x) \subseteq B_{r^\star}(y^\star).$
    Thus, we conclude $x \in W_{\rho_1}^c$ by
    \Cref{lem:NoWitnessesInStarBall}.
    In particular, we have shown 
    \begin{equation}\label{eqn:WaveOut}
        B_{\tilde{r}}(y_0) \cap W_{\rho_1}^c \ne \emptyset.
    \end{equation}
    The validity of {\bf Claim 2} thus follows from {\bf H2,}
    along with \Cref{eqn:WaveIn} and \Cref{eqn:WaveOut}.\\
    
    To conclude the proof of the Theorem, 
    we use {\bf Claim 2} to fix a point 
    $x \in B_{\tilde{r}}(y_0) \cap \partial W_{\rho_1}$
    and corresponding sequences converging to $x,$ namely
    \[
        (u_n) \subset W_{\rho_1} \cap B_{\tilde{r}}(y_0) \qquad \text{and} \qquad
        (w_n) \subset \big(W_{\rho_1}\big)^c \cap B_{\tilde{r}}(y_0).
    \]
    From {\bf Claim 1,} we conclude that $\clr(u_n) < \rho_2$ for all $n,$
    while $\clr(w_n) \ge \rho_1 > \rho_2$ for all $n.$ 
    This proves that $\clr$ is discontinuous at $x$ and the 
    result follows from \Cref{thm:DiscontinuousOnE}.
\end{proof}

\begin{remark}
We note that in \Cref{prop:H1Disc,thm:EnvGenThm} above, if it happens that
there is an $r^\star > 0$ with $B_{r^\star}(y^\star) \subset \xfree,$
then we can remove the assumption regarding 
$\F{\rho}(y_0) \cap \xfree \ne \emptyset,$ which is thus a consequence
of \Cref{lem:UnifPenetration}.
\end{remark}
\begin{ex}
We conclude the paper with an additional collection of examples,
demonstrating possible applications of \Cref{thm:EnvGenThm}.
\begin{enumerate}
    \item {\bf (Isolated Obstacle Points)} 
    Suppose $y_0 \in \xobst$ is an isolated point and 
    there exists $\xi \in \xspace$ so that
    \[
        h_F(y_0,\xi) := \inf_{v\in F(y_0)} \langle v, \xi \rangle > 0.
    \]
    Then the fact that $y_0$ is an envelope generator follows 
    from either \Cref{thm:EnvGenThm} or
    \Cref{thm:PersistentBddry}. 
    In the former approach, we use the fact that $y_0$ is 
    isolated to select $r^\star>0$ small enough so that
    $B_{r^\star}(y^\star) \subset \xfree,$ then apply 
    \Cref{thm:EnvGenThm}.
    Meanwhile in the latter approach, we first apply
    \Cref{thm:PersistentBddry}, then note that 
    given $r, \rho >0$ sufficiently small, it holds that
    $W_\rho \cap B_{r}(y_0) = \R{\rho}(y_0).$
    
    \vspace{1em}
    
    \item {\bf (Dubin's Car: Sharp Corner)}
    Working in $\xspace = \mathbb{R}^2\times \mathbb{T},$ where 
    $\mathbb{T} = [-\pi, \pi]$ is the flat torus with end points
    identified and topology generated by the periodic metric
    \[d_{\mathbb{T}}(a,b) := \min\{|a-b|, 2\pi - |a-b|\}.\]
    We consider the control system
    \begin{equation}\label{eqn:DubinsSystem}
    \begin{cases} \dot{x}_1 = \cos(x_3) \\ 
        \dot{x}_2 = \sin(x_3) \\
        \dot{x}_3 = u \end{cases}
    \qquad \text{for} \qquad u \in \uspace := [-1,1],
    \end{equation}
    along with a simple physical sharp corner obstacle at every $x_3$--level,
    \begin{equation}\label{eqn:DubinsCorner}
        \xobst := (-\infty,0] \times (-\infty,0] \times \mathbb{T}
        \qquad \text{and} \qquad \xfree := \xspace \setminus \xobst.
    \end{equation}
    The following analysis works for any choice of positive, continuous 
    running cost $\psi > 0.$
    
    \medskip
    
    \noindent{\bf Claim:} The point $y_0(\theta) = (0,0,\theta)$ is an 
    envelope generator for \Cref{eqn:DubinsSystem}--\Cref{eqn:DubinsCorner}
    if and only if $\theta \in [-\pi/2,0] \cup [\pi/2, \pi].$\\
    
    \noindent First, observe that if $\theta \in (0,\pi/2),$ then 
    $y-F(y) := \{y-v: v\in F(y)\} \subset \big(\xobst\big)^\circ$ 
    for all $y \in \xobddry$ sufficiently close to $y_0.$ 
    Thus, the only trajectories in $\xspace$ that reach these 
    points $y$ are inadmissible, as they must pass through the 
    interior of the obstacle space. Since this holds for all
    $y \in \xobddry$ near $y_0,$ it follows that $y_0$ is 
    bounded away from arbitrarily small waves.
    Meanwhile, if $\theta \in (-\pi,-\pi/2),$ it holds that 
    $x+F(x) := \{x+v: v\in F(x)\} \subset \big(\xobst\big)^\circ$ 
    for all $x \in \xspace$ sufficiently close to $y_0.$
    So, $\clr$ is continuous in a neighborhood of $y_0$
    (\Cref{fig:DubinsQIII}).
    Therefore, we conclude that $y_0$ cannot be an envelope generator 
    whenever $\theta \in (-\pi,-\pi/2) \cap (0,\pi/2).$\\
    
\begin{figure}[tbhp]
\centering
\subfloat[$-\pi<\theta<-\frac{\pi}{2}$]{
\includegraphics[clip=true,trim=0.75in 1.75in 0.25in 2.5in,height=2.5in]{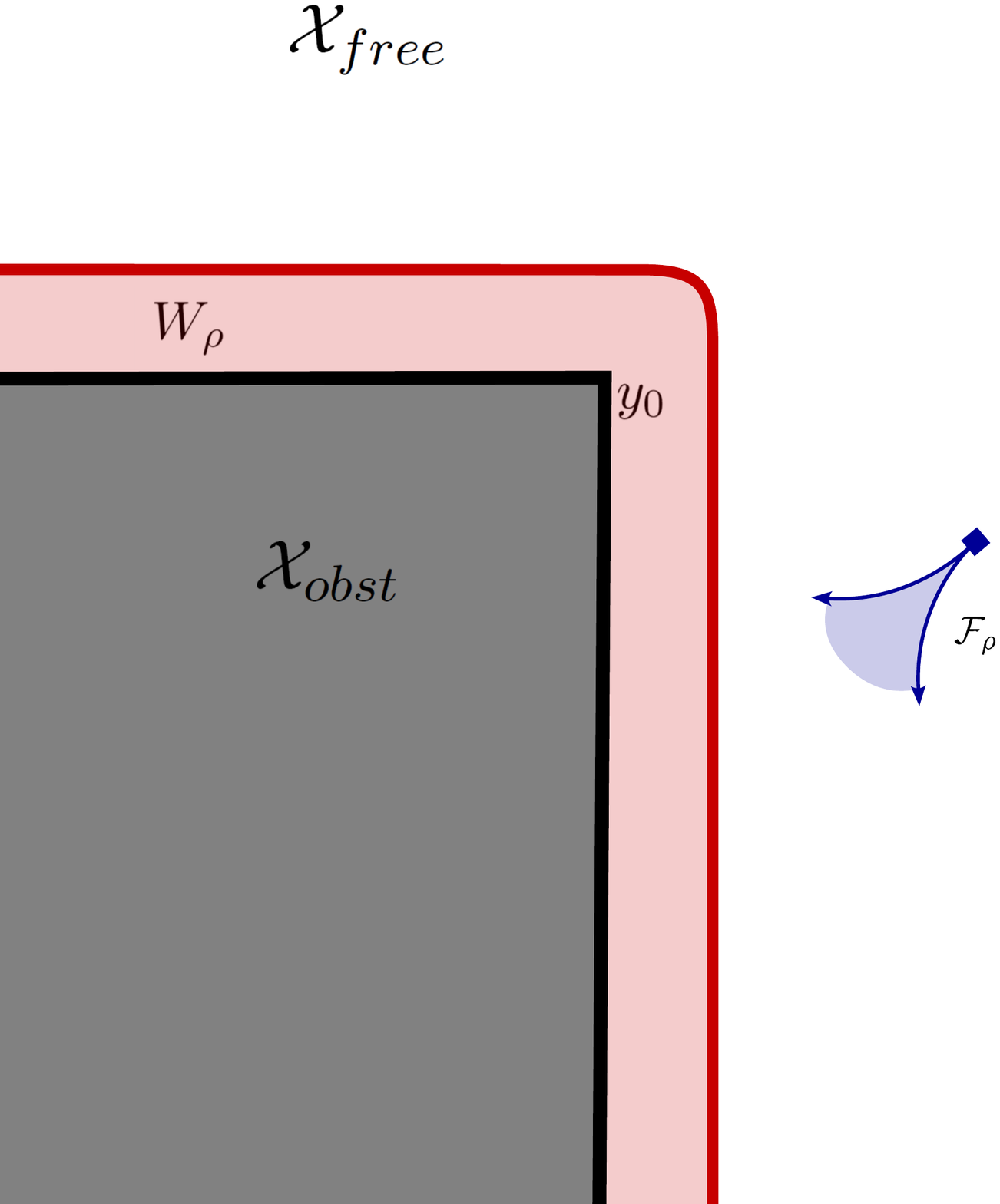}\label{fig:DubinsQIII}}
\hspace{.05cm}
\subfloat[$\frac{\pi}{2} \le \theta < \pi$]{
\includegraphics[clip=true,trim=0.75in 1.75in 0.25in 2.5in,height=2.5in]{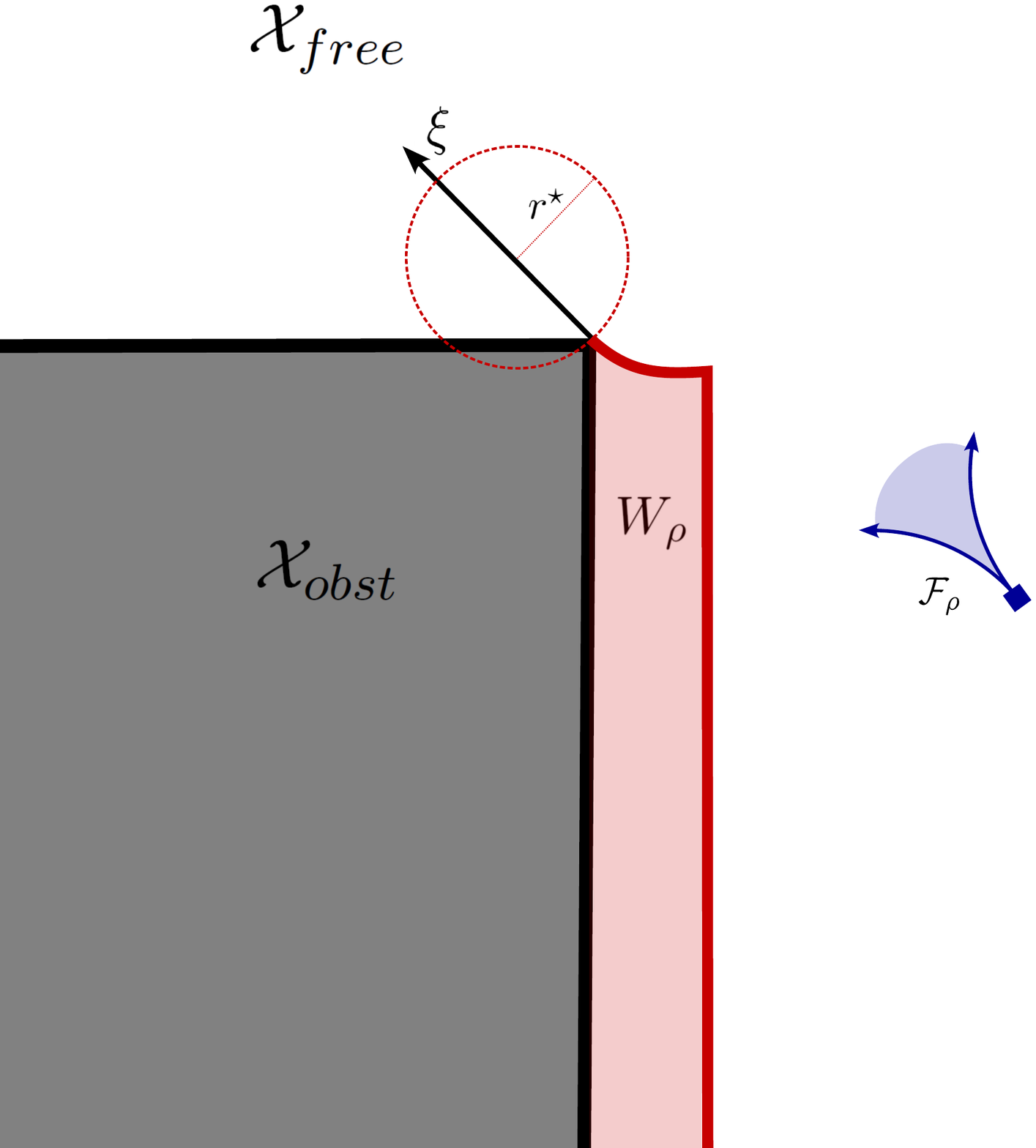}\label{fig:DubinsQII}}
\hspace{.05cm}
\subfloat[$-\frac{\pi}{2} < \theta \le 0$]{
\includegraphics[clip=true,trim=0.75in 1.75in 0.25in 2.5in,height=2.5in]{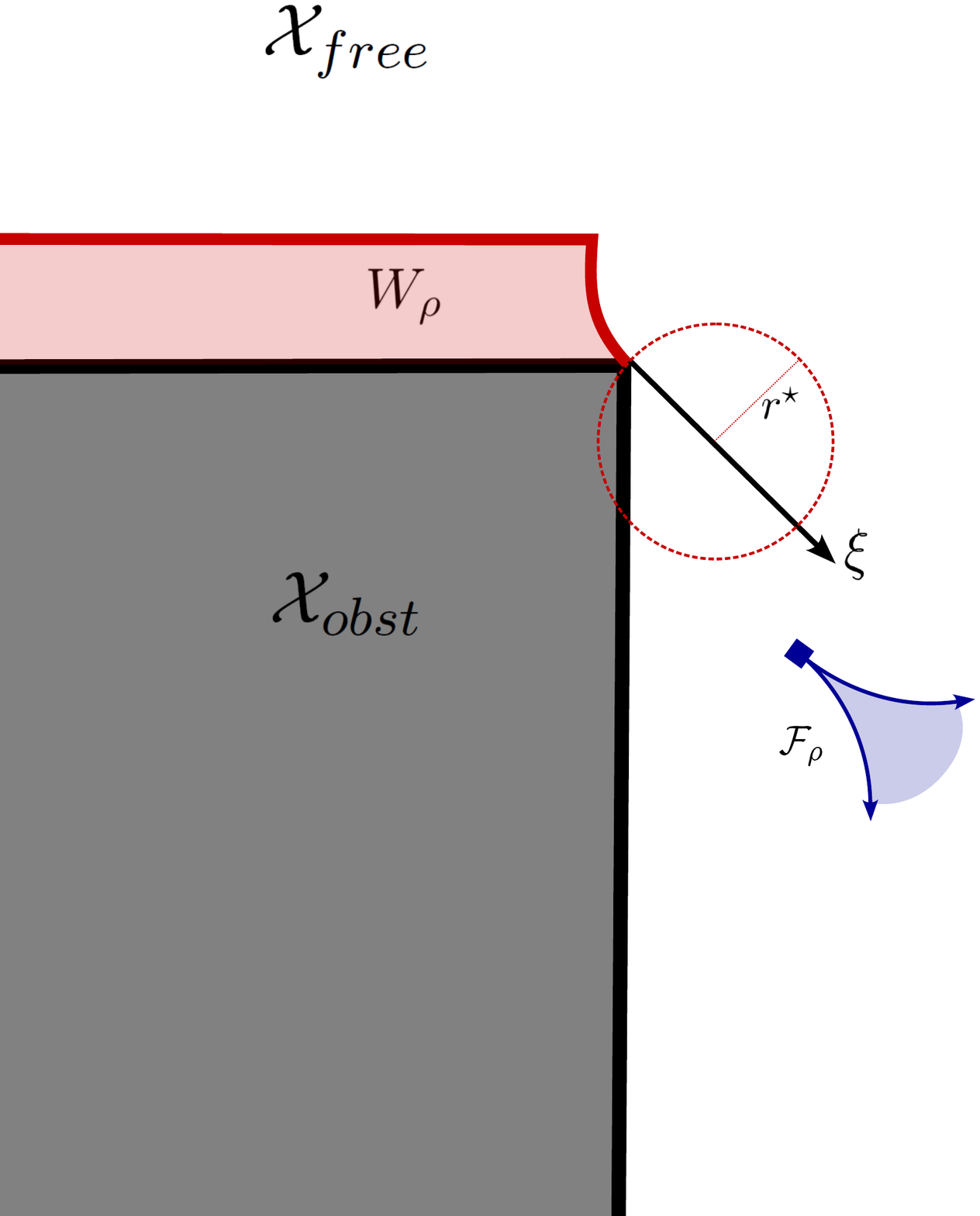}\label{fig:DubinsQIV}}
\caption{Dubin's car system, with sharp corner obstacle, and
initial orientation angles $\theta$ in third, 
second, and fourth quadrants, respectively. Example propagating 
waves present in each setting, along with a forward
reachable set for reference.}
\label{fig:DubinsExample}
\end{figure}
    
    Next, we fix $\theta \in [\pi/2, \pi)$ and choose the vector 
    \[
        \xi := (\cos(\theta), \sin(\theta), 0) \in \xspace.
    \]
    Since every $v \in F(y_0)$ has the form $v = (\cos(\theta), \sin(\theta), u),$
    it follows that $h_F(y_0, \xi) = 1 > 0.$
    Next, we fix $r^\star > 0$ sufficiently small so that
    \begin{equation}\label{eqn:y3Bound}
        y_3 < \pi \qquad \text{for all} \quad
        y = (y_1,y_2,y_3) \in B_{r^\star}(y^\star) 
            := B_{r^\star}\left(y_0 + r^\star \xi\right).
    \end{equation}
    It follows that every $y \in \xobddry \cap B_{r^\star}(y^\star)$
    must have the form (\Cref{fig:DubinsQII})
    \[
        y = (y_1,0,y_3) \qquad \text{for some $y_3 \in (0,\pi).$}
    \]
    Given any trajectory $\pi = (\pi_1,\pi_2,\pi_3) \in \Pi(\cdot,y)$
    propagating to such a point $y,$ it must hold that $\pi_3(t) \in (0,\pi)$
    and so $\dot{\pi}_2(t) = \sin(\pi_3(t)) > 0$ for all $t$ sufficiently
    close to $T_\pi.$ However, this means that $\pi_2(t) < 0$ and so
    $\pi(t) \in \big(\xobst\big)^\circ,$ for all such $t.$ 
    This proves that $\R{\rho}(y) \cap \xfree = \emptyset$
    for small $\rho > 0.$
    We conclude that assumption {\bf H1} holds at $y_0.$ 
    The remaining assumptions for \Cref{thm:EnvGenThm} 
    are straightforward to confirm, and so we conclude that 
    $y_0$ is an envelope generator for all $\theta \in [\pi/2, \pi).$
    
    Finally, we note that \Cref{thm:EnvGenThm} is not 
    directly applicable when $\theta = \pi;$ 
    owing to the fact that \Cref{eqn:y3Bound} is violated in 
    every $B_{r^\star}(y^\star)$ ball, even if one attempts to adjust 
    the choice of $\xi.$ 
    However, we can salvage the result by the fact
    that $y=(0,0,\pi)$ is a limit point of envelope generators. 
    Indeed, for any $r, \rho > 0$, we can select 
    $\theta' \in [\pi/2,\pi)$
    and $r' > 0$ so that $B_{r'}(y_0(\theta')) \subset B_r(y_0(\pi)).$
    It follows from $y_0(\theta')$ an envelope generator that 
    there exists $x \in E \cap B_{r'}(y_0(\theta))$ with 
    $\clr(x) < \rho.$ 
    Thus proving that $y_0(\pi)$ is itself an envelope generator.
    
    A similar argument proves the claim for $\theta \in [-\pi/2,0].$ 
    
    \vspace{1em}
    
    \item {\bf (System Admitting Horizontal Motion)}
    We conclude with an example demonstrating the fallacy of 
    the converse to \Cref{thm:EnvGenThm}.
    Working in $\xspace = \mathbb{R}^2,$ we consider the control system
   \begin{equation}\label{eqn:HorzSystem}
    \begin{cases} \dot{x}_1 = \cos(u) \\ 
        \dot{x}_2 = \sin(u) \end{cases}
    \qquad \text{for} \qquad u \in \uspace := [0,\pi],
    \end{equation}
    along with the simple sharp corner obstacle, 
    \begin{equation}\label{eqn:DubinsCornerAgain}
        \xobst := (-\infty,0] \times (-\infty,0] 
        \qquad \text{and} \qquad \xfree := \xspace \setminus \xobst,
    \end{equation}
    and $\psi \equiv 1.$
    
    For all $x = (x_1,x_2) \in \xfree,$ it is straightforward
    to confirm that
    \[
        \clr(x) = \begin{cases} x_1 & \text{if $x_2 \le 0$}\\
            \infty & \text{if $x_2 > 0.$} \end{cases}
    \]
    It follows that $y_0 = (0,0)$ is an envelope generator in
    this case, but we note that {\bf (H1)} fails to hold here.
\end{enumerate}
\end{ex}



\begin{thebibliography}{99}
\bibitem{AF20} P--C Aubin--Frankowski,
Lipschitz regularity of the minimum time function of differential inclusions with state constraints.
{\it Systems Control Lett.} {\bf 139} (2020), 104677, 7 pp.

\bibitem{BP87} G. Barles, B. Perthame,
Discontinuous solutions of deterministic optimal stopping time problems.
{\it RAIRO Mod\'el. Math. Anal. Num\'er.} {\bf 21} (1987), no. 4, 557--579.

\bibitem{BP07} A. Bressan, B. Piccoli,
{\it Introduction to the mathematical theory of control.}
AIMS Series on Applied Mathematics, 2. 
American Institute of Mathematical Sciences (AIMS), Springfield, MO, 2007.

\bibitem{CMN15} P. Cannarsa, A. Marigonda, K. Nguyen,
Optimality conditions and regularity results for time optimal control problems with differential inclusions.
{\it J. Math. Anal. Appl.} {\bf 427} (2015), no. 1, 202--228.

\bibitem{CMW12} P. Cannarsa, F. Marino, P. Wolenski,
Semiconcavity of the minimum time function for differential inclusions.
{\it Dyn. Contin. Discrete Impuls. Syst. Ser. B Appl. Algorithms} {\bf 19}
(2012), no. 1--2, 187--206.

\bibitem{CS15} P. Cannarsa, T. Scarinci,
Conjugate times and regularity of the minimum time function with differential inclusions.
Analysis and geometry in control theory and its applications, 
85--110, {\it Springer INdAM Ser.} {\bf 11,} Springer, Cham, (2015).

\bibitem{CQS97} P. Cardaliaguet, M. Quincampoix, P. Saint--Pierre,
Optimal times for constrained nonlinear control problems without local controllability.
{\it Appl. Math. Optim.} {\bf 36} (1997), no. 1, 21--42.

\bibitem{CLHKB05} H. Choset, K.M. Lynch, S. Hutchinson, G.A. Kantor, W. Burgard,
{\it Principles of robot motion: theory, algorithms, and implementations.}
MIT press, 2005.

\bibitem{CLSW95} F.H. Clarke, Yu.S. Ledyaev, R.J. Stern, P.R. Wolenski,
Qualitative properties of trajectories of control systems: a survey.
{\it J. Dynam. Control Systems} {\bf 1} (1995), no. 1, 1--48.

\bibitem{CLSW98} F.H. Clarke, Yu.S. Ledyaev, R.J. Stern, P.R. Wolenski,
{\it Nonsmooth analysis and control theory.}
Graduate Texts in Mathematics, 178. Springer--Verlag, New York, 1998.

\bibitem{CN10} G. Colombo, K.T. Nguyen,
On the structure of the minimum time function.
{\it SIAM J. Contol Optim.} {\bf 48} (2010), no. 7, 4776--4814.

\bibitem{dgta-marrt-14} J. Denny, E. Greco, S. Thomas, N. Amato,
MARRT: Medial axis biased rapidly--exploring random trees.
{\it 2014 IEEE International Conference on Robotics and Automation (ICRA)}
(2014), 90--97.

\bibitem{FPT11} A. Fedotov, V. Patsko, V. Turova,
Reachable sets for simple models of car motion. Chapter in
{\it Recent advances in mobile robotics.}
IntechOpen, Croatia, 2011.

\bibitem{FN15} H. Frankowska, L.V. Nguyen,
Local regularity of the minimum time function.
{\it J. Optim. Theory Appl.} {\bf 164} (2015), no. 1, 68--91.

\bibitem{kf-sbaomp-11} S. Karaman, E. Frazzoli,
Sampling--based algorithms for optimal motion planning.
{\it The International Journal of Robotics Research} 
{\bf 30} (2011), no. 7, 846--894.

\bibitem{l-pa-06} S.M. LaValle,
{\it Planning Algorithms.} 
Cambridge University Press, USA, 2006.

\bibitem{L93} P.D. Loewen,
{\it Optimal control via nonsmooth analysis.}
CRM Proceedings \& Lecture Notes {\bf 2}.
American Mathematical Society, Providence, RI, 1993.

\bibitem{N16} L.V. Nguyen,
Variational analysis and regularity of the minimum time function for differential inclusions.
{\it SIAM J. Control Optim.} {\bf 54} (2016), no. 5, 2235--2258.

\bibitem{N07} C. Nour,
Existence of solutions to a global eikonal equation.
{\it Nonlinear Anal.} {\bf 67} (2007), no. 2, 349--367.

\bibitem{NS07} C. Nour, R.J. Stern,
Regularity of the state constrained minimal time function.
{\it Nonlinear Anal.} {\bf 66} (2007), no. 1, 62--72.

\bibitem{N92} R.L. Nowack,
Wavefronts and solutions of the eikonal equation.
{\it Geophysical Journal International} {\bf 10} (1992), iss. 1, 55--62.
  
\bibitem{sjp-osbmpudc-15} E. Schmerling, L. Janson, M. Pavone,
Optimal sampling--based motion planning under differential constraints: the driftless case.
{\it 2015 IEEE International Conference on Robotics and Automation (ICRA)} (2015), 2368--2375.

\bibitem{V97} V.M. Veliov,
Lipschitz continuity of the value function in optimal control.
{\it J. Optim. Theory Appl.} {\bf 94} (1997), no. 2, 335--363.

\bibitem{WZ98} P.R. Wolenski, Y. Zhuang,
Proximal analysis and the minimal time function.
{\it SIAM J. Control Optim.} {\bf 36} (1998), no. 3, 1048--1072.
\end{thebibliography}
\end{document}